\documentclass[11pt,a4paper,twoside]{article}
\usepackage{color}
\usepackage{bm, amsmath, amssymb, amsthm} 

\def\calf{{\cal F}}
\def\<{\langle}
\def\>{\rangle}
\def\eps{\varepsilon}

\def\RR{\mathbb{R}}

\newcommand\const{\operatorname{const}}

\newcommand\Div{\operatorname{div}}
\newcommand\id{\operatorname{id}}
\newcommand\im{\operatorname{im}}

\def\eq{\hspace*{-.8mm}&=&\hspace*{-.8mm}}
\def\plus{\hspace*{-.8mm}&+&\hspace*{-.8mm}}
\def\minus{\hspace*{-.8mm}&-&\hspace*{-.8mm}}

\newcommand{\gv}{\mathrm{gv}}

\newtheorem{cor}{Corollary}
\newtheorem{definition}{Definition}
\newtheorem{example}{Example}
\newtheorem{remark}{Remark}
\newtheorem{lem}{Lemma}
\newtheorem{prop}{Proposition}
\newtheorem{thm}{Theorem}

\author{Vladimir Rovenski\footnote{Mathematical Department, University of Haifa, Mount Carmel, 31905 Haifa,  Israel
        \newline e-mail: {\tt vrovenski@univ.haifa.ac.il}        }
        \ and \
        Pawe\l \  Walczak\footnote{Katedra Geometrii,
        Uniwersytet \L\'{o}dzki, ul. Banacha 22,
             90-238  \L\'{o}d\'{z}, Poland
        \newline e-mail: {\tt pawelwal@math.uni.lodz.pl}
}
}

\title{A Godbillon-Vey type invariant \\ for a 3-dimensional manifold with a~plane field}

\begin{document}

\date{}

\maketitle

\begin{abstract}
We consider a 3-dimensional smooth manifold $M$ equipped with an arbitrary,
\textit{a~priori} non-integrable, distribution (plane field) ${\cal D}$ and a vector field $T$ transverse to ${\cal D}$.
Using a 1-form $\omega$ such that ${\cal D} = \ker\,\omega$
and $\omega(T)=1$ we construct a 3-form analogous to that defining the Godbillon-Vey class of a foliation,
and show how does this form depend on $\omega$ and~$T$.
For a compatible Riemannian metric on $M$, we express this 3-form in terms of the curvature and torsion of normal curves
and the non-symmetric second fundamental form of ${\cal D}$.
We deduce Euler-Lagrange equations of associated functionals: for variable $({\cal D},T)$ on $M$,
and for variable Riemannian or Randers metric on $(M,{\cal D})$.
We show that for a geodesic field $T$ (e.g., for a contact structure) such $({\cal D},T)$ is critical,
characterize critical pairs when ${\cal D}$ is integrable, and prove that these critical pairs
are not extrema.


\end{abstract}

\section*{Introduction}

The Godbillon-Vey cohomology class $\gv(\calf)$ of a transversely oriented codimension-one foliation
$\calf$ of a compact manifold $M$ has been defined in \cite{gv} as the de Rham cohomology class of the 3-form
$\eta\wedge d\eta$, where $\eta$ is a 1-form satisfying $d\omega =\omega\wedge\eta$,
$\omega$ being a 1-form defining the tangent bundle (distribution $T\calf = \ker \omega$) of $\calf$.
If $\dim M = 3$ then $\gv(\calf)$ provides a number:
\begin{equation}\label{E-gv-invar0}
 \gv(\calf)=\int_M \eta\wedge d\eta.
\end{equation}
Integrability of $T\calf$ implies the existence of such $\eta$  while non-complicated calculations show that
 $\gv(\calf)$ is well defined, that is it does not depend on possible choices of the forms $\omega$ and~$\eta$.
The~Godbil\-lon-Vey class plays a crucial role in the study of  topology and dynamics of
foliations and, still, is of some interest among ``foliators'', see
e.g. \cite{cc,glw,hu2002,hl,wa} and \cite[Problem~10]{hu2005}.

Now, let $g$ be a Riemannian metric on $M$ ($\dim M = 3$), $\nabla$ its Levi-Civita connection,
$T$ the positive oriented  unit vector field on $M$ normal to $\calf$ and $h$ the scalar second fundamental form.
Denote by $k$ the function on $M$ such that $k(x)$ is the curvature of the $T$-curve at $x\in M$
and assume that $k\ne0$ on an open set $U$ of~$M$.
Thus, the unit normal $N$, the binormal $B=T\times N$ and the \textit{torsion} $\tau$ of $T$-curves are defined on $U$.
Then
\[
 \eta\wedge d\eta=k^2(\tau - h_{B,N})\,{\rm d}V_g\ \ {\rm on}\ \ U\ \ {\rm and}\ \ \eta\wedge d\eta=0\ \ {\rm when}\ \ k=0;
\]
hence, the~following formula holds \cite{rw73}:
\begin{equation}\label{eq:rwood0}
 \gv(\calf) = \int_M k^2(\tau - h_{B,N})\,{\rm d}V_g.
\end{equation}

In this paper, we consider a 3-dimensional compact manifold $M$ equipped with an arbitrary,
\textit{a~priori} non-integrable, distribution (plane field) ${\cal D}$ and a vector field $T$ transverse to ${\cal D}$.
Non-integrable distributions appear in many situations, e.g. on contact manifolds and in sub-Riemannian geometry.
Using a 1-form $\omega$ such that ${\cal D} = \ker\,\omega$ and $\omega(T)=1$,
we construct a 3-form analogous to that defining $\gv(\calf)$ in \eqref{E-gv-invar0}.
In a sense, our form arises from the best approximation of $d\omega$ by the wedge-product of $\omega$ by a 1-form.
We~show how does this form depend on $\omega$ and~$T$.
On a Riemannian manifold, this form is expressed in terms of the curvature and
torsion of normal curves and the non-symmetric second fundamental form of ${\cal D}$, see \eqref{eq:rwood0} for foliations.
We provide also variational formulas related to our construction and then deduce Euler-Lagrange equations of associated functional:
for variable pair $({\cal D},T)$ on $M$ (with fixed adapted metric),
and for variable Riemannian (or Randers) metric $g$ on $(M,{\cal D})$.
 We characterize critical pairs when ${\cal D}$ is integrable,
show that for a geodesic field $T$ (e.g. contact structure) such $({\cal D},T)$ is critical;
and prove that these critical pairs have saddle type.
In the last section we discuss Godbillon-Vey invariants on a manifold of dimension $2n+1\ge5$.

\section{Construction}
\label{sec:1-1}

Let ${\cal D}$ be a transversely oriented codimension one distribution on a 3-dimensional smooth compact manifold $M$,
and $\omega$ a 1-form such that ${\cal D} = \ker\,\omega$. Let $T$ be a vector field on $M$ such that
\[
 \omega(T)=1.
\]
For any 1-form $\alpha$ and vector fields $X,Y\in\mathfrak{X}_M$ on $M$ we use the formulae
\[
 d\alpha(X,Y) = X(\alpha (Y)) -Y(\alpha (X))-\alpha ([X,Y]),\quad
 \iota_{\,Z}\,d\omega = d\omega(Z,\,\cdot),\quad {\cal L}_{\,Z}\alpha = (d\iota_Z + \iota_Zd)\alpha .
\]
We also will apply the inner product $\iota_{\,Z}$ and the Lie derivative ${\cal L}_{\,Z}$ to other differential forms.

\begin{definition}\rm
A Riemannian metric $g$ on $M$ is \textit{compatible} with a pair $({\cal D},T)$ if $T$ is the unit normal to~${\cal D}$,
i.e. $\omega(X) = g(T, X)$ for all $X\in\mathfrak{X}_M$.
Denote by ${\rm Riem}(M,{\cal D},T)$ the space of all such metrics.
\end{definition}

 The ``musical" isomorphisms $\sharp$ and $\flat$
(called and denoted so because, in traditional notation, they ``lower'' and ``raise'' indices)
will be used for rank one tensors, e.g. $\omega(X)=g(\omega^\sharp,X)=X^\flat(\omega^\sharp)$.

 Given $g\in{\rm Riem}(M,{\cal D},T)$, consider in the space $\Lambda^1(M)$ of 1-forms on $M$ the subspace
$\omega^\perp$ ortho\-gonal to the line spanned by $\omega$.
Consider also, in the space $\Lambda^2 (M)$ of 2-forms, the subspace $\omega\wedge\omega^\perp$
of all the 2-forms $\omega\wedge\theta$, $\theta$ being a 1-form of $\omega^\perp$.
 Now, project $d\omega$ orthogonally onto the subspace $\omega\wedge\omega^\perp$.
The projection has the form $\omega\wedge\eta$, $\eta$ belonging to $\omega^\perp$. Such $\eta$ is unique.

\begin{lem}\label{L-eta}
The 1-form $\eta$ does not depend on a compatible metric $g$, and we have
\begin{equation}\label{E-eta}
 \eta = \iota_{\,T}\,d\omega,
 \quad{\rm equivalently,}\quad \eta = {\cal L}_{\,T}\,\omega.
\end{equation}
\end{lem}

\begin{proof} Let $g\in{\rm Riem}(M,{\cal D},T)$. The property $\eta\,\perp\,\omega$ means $\eta(\omega)=0$.
The property $d\omega-\omega\wedge\eta\ \perp\ \omega\wedge \omega^\perp$ means $(d\omega -\omega\wedge\eta)(T,\,\cdot)=0$,
that is $\iota_{\,T}\,d\omega-\eta=0$.
Note also that $d\eta = {\cal L}_{\,T}(d\omega)$.
\end{proof}

The 3-form $\eta\wedge d\eta$ defines the de Rham cohomology class
and represents the \textit{Godbillon-Vey type invariant} of a pair $({\cal D},T)$:
\begin{equation}\label{E-gv-invar}
 \gv: ({\cal D}, T) \to \int_M \eta\wedge d\eta.
\end{equation}


\begin{example}\label{Ex-cont}\rm
Let $T$ be the Reeb field of a contact distribution ${\cal D}=\ker\omega$, then $\omega(T)=1$ and
$\eta:=\iota_{\,T}\,d\omega$ vanishes, see definitions in \cite{b2010}.
\end{example}

Next proposition shows how does $\gv({\cal D}, T)$ change when
$T$ and $\omega$ vary into $\tilde T$ and $\tilde\omega$, while condition $\omega(T)=1$ is preserved.
Using the fact that $\tilde T -T$ can be uniquely decomposed into $T$- and ${\cal D}$-components,
we have three different cases for
$(\tilde T,\tilde\omega)$, and ${\cal D}$ is preserved in cases (i) and~(ii):

\smallskip
(i)~$\tilde T$ is parallel to $T$ and $\tilde\omega$ is parallel to $\omega$,

\smallskip
(ii)~$\tilde T - T$ belongs to $\mathfrak{X}_{\cal D}$ (hence $\omega(\tilde T)=1$)
and $\tilde\omega=\omega$,

\smallskip
(iii)~$\tilde T = T$ and $\tilde\omega = \omega +\mu$ for some 1-form $\mu$ such that $\mu(T)=0$.

\begin{prop}\label{prop1}
Let $(\omega,T)$ and $(\tilde\omega,\tilde T)$ be pairs of smooth one-forms and vector fields on $M$ obeying
\begin{equation}\label{E-wt=1}
 \omega(T)=1=\tilde\omega(\tilde T).
\end{equation}

{\rm (i)} If $\,\tilde T = e^{-f} T$ for a smooth function $f$ on $M$ then $\tilde\omega=e^{f}\omega$
and the corresponding 3-forms $\eta\wedge d\eta$ and $\tilde\eta\wedge d\tilde\eta$ are related by
\begin{equation}\label{eq:rel_forms}
 \tilde\eta\wedge d\tilde\eta = \eta\wedge d\eta + d\alpha + 2\,T(f)\cdot\omega\wedge d\eta + T(f)^2\cdot\omega\wedge d\omega,
\end{equation}
$\alpha=f\,d\eta -f\,d(T(f)\,\omega) +T(f)\,\omega\wedge\eta$ being a 2-form.

{\rm (ii)} If $\,\tilde T=T+X$ for some $X\in\mathfrak{X}_{\cal D}$ and $\tilde\omega=\omega$ then $\tilde\eta=\eta +\iota_X\,d\omega$
and the corresponding 3-forms $\eta\wedge d\eta$ and $\tilde\eta\wedge d\tilde\eta$ are related by
\begin{equation}\label{eq:rel-forms2-gen}
 \tilde\eta\wedge d\tilde\eta
 = \eta\wedge d\eta + d(\eta\wedge\iota_X d\omega) + 2\,d\eta\wedge\iota_X d\omega + \iota_X(d\omega)\wedge d(\iota_X d\omega).
\end{equation}

{\rm (iii)} If $\,\tilde T=T$ and $\tilde\omega=\omega+\mu$ for some 1-form $\mu$ then $\mu(T)=0$,
$\tilde\eta=\eta +\iota_{\,T}\,d\mu$ and the corresponding 3-forms $\eta\wedge d\eta$ and $\tilde\eta\wedge d\tilde\eta$ are related by
\begin{equation}\label{eq:rel-forms3-gen}
 \tilde\eta\wedge d\tilde\eta
 = \eta\wedge d\eta + d(\eta\wedge\iota_{\,T}\,d\mu) + 2\,d\eta\wedge \iota_T\,d\mu + \iota_{\,T}\,d\mu\wedge d(\iota_{\,T}\,d\mu).
\end{equation}
\end{prop}

\begin{proof}
(i) In this situation, \eqref{E-wt=1} yields $\tilde\omega=e^{f}\omega$.
By Lemma~\ref{L-eta}, we find
\begin{eqnarray*}
  && \tilde\eta=e^{-f}\iota_T(d e^{f}\omega) = e^{-f}\iota_T( e^{f} df\wedge\omega + e^{f} d\omega)
  = \iota_T( df\wedge\omega + d\omega) = \eta - df + T(f)\,\omega,\\
 && d\tilde\eta = d\eta + d(T(f))\wedge\omega + T(f)\,d\omega.
\end{eqnarray*}
The last two equalities imply (\ref{eq:rel_forms}).

(ii) Certainly, \eqref{E-wt=1} and $\tilde\omega=\omega$ yield $\omega(X)=0$.
Using Lemma~\ref{L-eta} yields $\tilde\eta = \eta +\iota_X(d\omega)$.
From the above \eqref{eq:rel-forms2-gen} follows.

(iii) This situation is dual to (ii), in a sense. Certainly, \eqref{E-wt=1} and $\tilde T= T$ yield $\mu(T)=0$.
Using Lemma~\ref{L-eta} yields $\tilde\eta=\eta+\iota_{\,T}(d\mu)$.
From the above \eqref{eq:rel-forms3-gen} follows.
\end{proof}

\begin{remark}\rm
Formula \eqref{eq:rel-forms3-gen} cannot be reasonably simplified even in the integrable case:
two foliations defined by $\omega$ and $\tilde\omega$ may have different $\gv$-classes and then the forms $\eta\wedge d\eta$
and $\tilde\eta \wedge d\tilde\eta$ differ by a form which is not exact and difficult to express explicitly.
\end{remark}

If ${\cal D}$ is integrable, then $\omega\wedge d\omega=0$, see \cite{cc}, and $d\omega=\omega\wedge\eta$.
Derivation of the last equality yields
\[
 0= d(d\omega) = \eta\wedge d\omega - d\eta\wedge\omega = (\omega\wedge\eta)\wedge\eta - d\eta\wedge\omega = - d\eta\wedge\omega.
\]
Hence, the last two terms in (\ref{eq:rel_forms}) vanish when ${\cal D}$ is integrable, but in any case we have the following.

\begin{cor}
The cohomology class $\gv({\cal D}, T)$ does not change when we replace
$T$ by $\tilde T = e^{-f} T$ with $f$ constant along the $T$-curves, i.e. when $T(f)=0$.
\end{cor}

\section{Variations}
\label{sec:1-2}

For variable pairs $(\omega_t,T_t)$ or metrics $g_t$, denote by $\,^\centerdot{ }\,$ the $t$-derivative at $t=0$ of any quantity on~$M$.
As for $(\tilde\omega,\tilde T)$ before, we have three independent cases for a pair $(\omega_t,T_t)$
such that $\omega_t(T_t)\equiv1$,

\smallskip
(i)~$\dot T$ is parallel to $T$ and $\dot\omega$ is parallel to $\omega$,

\smallskip
(ii)~$\dot T\in\mathfrak{X}_{\cal D}$ and $\dot\omega=0$, hence $\omega(\dot T)=0$,

\smallskip
(iii)~$\dot T = 0$ and $\dot\omega$ is a 1-form such that $\dot\omega(T)=0$.


\begin{lem} Let $(\omega_t,T_t)\ (|t|\le\eps)$ be a smooth family of pairs of one-forms and vector fields on $M^3$ satisfying
$\omega_t(T_t)\equiv1$ and let ${\cal D}_t=\ker\omega_t$.
Then
\begin{eqnarray}\label{E-gv-omegat}
 && \dot\gv := \gv({\cal D}_t,T_t)'_{\,|\,t=0}
 = 2\int_M \dot\eta\wedge d\eta\,, \\
 \label{E-gv-omegatt}
 && \ddot\gv := \gv({\cal D}_t,T_t)''_{\,|\,t=0} = 2\int_M (\ddot\eta\wedge d\eta + \dot\eta\wedge d\dot\eta),
\end{eqnarray}
where $\eta=\iota_{\,T}\,d\omega$, $\dot\eta=\iota_{\,T}\,d\dot\omega + \iota_{\,\dot T}\,d\omega$
and $\ddot\eta = \iota_{\,T}\,d\ddot\omega + \iota_{\,\dot T}\,d\dot\omega + \iota_{\,\ddot T}\,d\omega$.
\end{lem}

\begin{proof}
From the Taylor expansions
\[
 \omega_t = \omega + t\dot\omega+(t^2/2)\ddot\omega +O(t^3),\quad T_t = T + t\dot T +  (t^2/2)\ddot T +O(t^3),
\]
we obtain
 $\dot\omega(T) + \omega(\dot T) = 0$
 and
 $\omega (\ddot T) + 2\,\dot\omega (\dot T) + \ddot\omega (T) = 0$.
Let $\eta_t =  \iota_{\,T_t}\,d\omega_t $. Write $\eta_t = \eta + t\,\dot\eta + (t^2/2)\ddot\eta +O(t^3)$.
Then
 $\dot\eta = \iota_{\,T}\,d\dot\omega + \iota_{\,\dot T}\,d\omega$
 and
 $\ddot\eta = \iota_{\,T}\,d\ddot\omega + 2\,\iota_{\,\dot T}\,d\dot\omega + \iota_{\,\ddot T}\,d\omega$.
Let
 $\gv (t) = \int_M \eta_t\wedge d\eta_t$,
and write $\gv (t) = \gv + t\,\dot\gv + (t^2/2)\ddot\gv + O(t^3)$.
Since
\[
 \eta_t\wedge\,d\eta_t =\eta\wedge\,d\eta + t(\eta\wedge\,d\dot\eta +\dot\eta\wedge\,d\eta)
 +t^2(\ddot\eta\wedge d\eta + \dot\eta\wedge d\dot\eta)+ O(t^3),
\]
using $d(\dot\eta\wedge\,\eta)=d\dot\eta\wedge\eta -\dot\eta\wedge d\eta$
and the Divergence Theorem, we obtain \eqref{E-gv-omegat} and \eqref{E-gv-omegatt}.
\end{proof}

Substituting the formulae for $\dot\eta$ and $\ddot\eta$ into that for $\ddot\gv$
we obtain a general formula for the second variation of our Godbillon-Vey invariant.
It should be interesting at critical points of $\gv$, in particular at these with $d\eta=0$, where the stability condition reduces to $I(\dot\eta, \dot\eta)\ge 0$ (or, $\le 0$) for any 1-form $\dot\eta$ as above, where the symmetric bilinear form $I$ on the space of 1-forms is given by
 $I(\phi, \psi) = \int_M \phi\wedge d\psi$.

\begin{thm}
If ${\cal D}$ is integrable then:
$(a)$~a pair $({\cal D},T)$ is critical for \eqref{E-gv-invar} if and only if
\begin{equation}\label{E-ELiii}
 ({\cal L}_{\,T})^3\omega = 0.
\end{equation}
$(b)$~there are no extremal with respect to \eqref{E-gv-invar} two-dimensional foliations on $M^3$.
\end{thm}

\begin{proof}
When ${\cal D}$ is tangent to a foliation, variations of types (i)--(ii) do not change the functional,
and only variations of type (iii) are essential.

(a)~In this case $({\cal D}_t,T)$, we have $\dot\eta=\iota_T\,d\dot\omega$ and $\dot T=0$ with $\dot\omega(T)=0$. Thus \eqref{E-gv-omegat} reads
\begin{eqnarray*}
 && \gv({\cal D}_t,T)'_{\,|\,t=0}
 = 2\int_M (\iota_T\,d\dot\omega)\wedge d\eta
 = 2\int_M d\dot\omega\wedge \iota_T\,d\eta
 = 2\int_M \dot\omega\wedge d(\iota_T\,d\eta).
\end{eqnarray*}
For critical pair $({\cal D},T)$, the above yields point-wise equality
 $\dot\omega\wedge d(\iota_T\,d\eta)=0$.
Given $X,Y\in\mathfrak{X}_{\cal D}$ (linearly independent at a point $x$), take $\dot\omega$ such that $\dot\omega(X)=0$ and $\dot\omega(Y)=1$.
Then, using $\eta=\iota_{\,T}\,d\omega={\cal L}_{\,T}\,\omega$, we get
\begin{eqnarray*}
 0 = \dot\omega\wedge d(\iota_T\,d\eta)(T,X,Y) = d(\iota_{\,T}\,d\eta)(T,X)
 \ \Rightarrow \ \iota_{\,T}\,d(\iota_T\,d\eta)=0\ \Leftrightarrow \
 ({\cal L}_{\,T})^2\eta =0.
\end{eqnarray*}
By this and definition of $\eta$, we obtain \eqref{E-ELiii}.

(b)~First, we claim that a critical pair $({\cal D},T)$ is an extremum for
\eqref{E-gv-invar} if and only if the
bilinear~form
\[
 I_T(\alpha,\beta)=\int_M({\cal L}_T^2\,d\alpha)\wedge\beta,
\]
where $\alpha,\beta$ are 1-forms on $M$ obeying $\alpha(T)=\beta(T)=0$, is definite.
Indeed, since $\ddot\eta=\iota_T\,d\ddot\omega$ and $\ddot T=0$ with $\ddot\omega(T)=0$,
we get $\int_M\ddot\eta\wedge d\eta=0$, thus \eqref{E-gv-omegatt} reads
\[
 \gv({\cal D}_t,T)''_{\,|\,t=0} = 2\int_M \dot\eta\wedge d\dot\eta
 = 2\int_M \iota_{\,T}\,d\dot\omega\wedge d(\iota_{\,T}\,d\dot\omega)
 = 2\int_M ({\cal L}_T^2\,d\dot\omega)\wedge\dot\omega
 = 2\,I_T(\dot\omega,\dot\omega).
\]
It is easy to check that the bilinear form $I_T(\alpha,\beta)$ is symmetric. Thus the claim follows.

Let $T=\partial_z$ on a domain in $M$ with coordinates $(x_1,x_2,x_3=z)$.
Then $\mu(T)=0$ for any 1-form $\mu=p_1(x_1,x_2,z)dx^1+p_2(x_1,x_2,z)dx^2$ where $p_1,p_2$ are supported in the coordinate domain.
We~calculate
\begin{equation}\label{E-ind-a}
 {\cal L}_T\,\mu = T(p_1)\,dx^1+T(p_2)\,dx^2,\quad
 ({\cal L}_T)^2\mu = T^2(p_1)\,dx^1+T^2(p_2)\,dx^2
 = p_{1,33}\,dx^1 + p_{2,33}\,dx^2,
\end{equation}
and then
\[
 ({\cal L}_T)^2d\mu = d({\cal L}_T)^2\mu =
 \big( p_{2,133}-p_{1,233} \big)\,dx^1\wedge dx^2 - p_{1,333}\,dx^1\wedge dz - p_{2,333}\,dx^2\wedge dz.
\]
Hence
\[
 I_T(\mu,\mu)= \int_M\big( p_{1,333}\,p_2 - p_{2,333}\,p_1\big)dx^1\wedge dx^2\wedge dz.
\]
This quadratic form may have any sign.
\end{proof}

\begin{lem}\label{L-warped}
Let ${\cal D}=\ker\omega$ be tangent to a totally umbilical foliation of a Riemannian manifold $(M^3,g)$.
Let $T$ be unit tangent vector to $T$-curves and $k,\tau$ their curvature and torsion.
Then \eqref{E-ELiii} is equivalent to the system, where $2\lambda=\sigma_1$ is the mean curvature of ${\cal D}$,
\begin{equation}\label{E-wp-system1}
 T(T(k)) = (\tau^2-\lambda^2) k + T(\lambda\,k) +\lambda\,T(k),\quad
 T(k\,\tau) + \tau T(k) -2\,\tau\lambda = 0.
\end{equation}
\end{lem}

\begin{proof}
The metric $g$ is compatible with a pair $({\cal D},T)$, and we will use some notations of Section~\ref{sec:RW-form}.
Let $(T,N,B)$ be a Frenet frame (on $U$) for $T$-curves.
Take any vector field $X$ in the distribution ${\cal D}$
such that $X_1=\<X,N\>$ and $X_2=\<X,B\>$ are constant.
We have $AX=\lambda X$ and
\begin{equation}\label{E-wp-1}
 [T, X] = \nabla_T X -\nabla_X T = - k\,X_1 T +(\lambda X_1-\tau X_2) N + (\tau X_1 +\lambda X_2) B .
\end{equation}
By definition of the Lie derivative,
\begin{eqnarray*}
 {\cal L}_T\,\omega(X) \eq \eta(X)=k X_1,\\
 {\cal L}^2_T\,\omega(X)\eq T({\cal L}_T\,\omega(X)) -{\cal L}_T\,\omega([T, X]) \\
 \eq T(k) X_1 -k\<[T, X],N\> = T(k) X_1 +k(\tau X_2 - \lambda X_1).
\end{eqnarray*}
Then we find
\begin{eqnarray*}
 {\cal L}_T\,\omega([T, X]) \eq k(\lambda X_1 -\tau X_2),\\
 {\cal L}^2_T\,\omega([T, X]) \eq T({\cal L}_T\,\omega([T,X]) - {\cal L}_T\,\omega([T,[T,X]])\\
 \eq (T(k)\lambda +k(\tau^2-\lambda^2)) X_1 +(2\lambda\,\tau - \tau T(k)) X_2.
\end{eqnarray*}
By the above, we obtain for arbitrary $X_1,X_2$:
\begin{eqnarray*}
 {\cal L}^3_T\,\omega(X) \eq T({\cal L}^2_T\,\omega(X)) -{\cal L}^2_T\,\omega([T, X])\\
 \eq (T(T(k)) - k(\tau^2-\lambda^2) -T(k\lambda) -T(k)\lambda) X_1 + (T(k\tau) + \tau T(k) -2\lambda\,\tau) X_2,
\end{eqnarray*}
and the system \eqref{E-wp-system1} follows.
\end{proof}

\begin{remark}\label{R-wp}\rm
The $(M^3,g)$ in Lemma~\ref{L-warped} is locally a double-twisted product, and $\lambda=0$ (totally geodesic foliation)
corresponds to a twisted product.
For $\lambda=0$, \eqref{E-ELiii} is equivalent to the system,
see \eqref{E-wp-system1},
\begin{equation}\label{E-wp-system}
 T(T(k)) = \tau^2 k,\qquad T(k\,\tau) + \tau\,T(k) = 0.
\end{equation}
If $\tau=0$ and $k=\const$ on $T$-curves then we have a solution.
If $\tau\ne0$ and $k\ne0$ then by \eqref{E-wp-system}$_2$, $\tau=c/k^2$, where $c$ is constant on $T$-curves,
and by \eqref{E-wp-system}$_1$ we find
 $T(T(k))= c^2/k^3$,
which is integrable.
Since $T(T(k))\ge0$ when $k>0$, the only periodic solutions of \eqref{E-wp-system} are $k(z)=\const$ and~$\tau=0$.
\end{remark}

Let $(B,g_B)$ and $(F,g_F)$ be Riemannian manifolds and $\phi>0$ a smooth function on $B\times F$.
The~\textit{twisted product} $M=B\times_\phi F$ is the manifold $M=B\times F$ with the
metric
 $g = \pi^*g_B + (\phi\circ \pi)^2(\pi')^*g_F$,
where $\pi:M\to B$ and $\pi':M\to F$ are projections.
The fibers $\{x\}\times F\ (x\in B)$ are totally umbilical,
and the leaves $B\times\{y\}\ (y\in F)$ are totally geodesic.
If we regard $\pi:B\times_\phi F\to B$ as a submersion, then the fibers are conformally related with each other;
this gives us a \textit{conformal submersion}. If $\phi$ depends on $B$ only, twisted product becomes the \textit{warped~product}.

\begin{lem}[see \cite{c-by}]\label{L-chen1}
Let $M=B\times_\phi F$ be a twisted product. Then

(i) fibers $\{x\}\times F$ are totally umbilical in $M$ with the mean curvature vector $-(\nabla\log\phi)^\top$,

(ii) fibers have parallel mean curvature if and only if $\phi=\phi_1\phi_2$ with
$\phi_1\in C^2(B)$ and $\phi_2\in C^2(F)$.
\end{lem}


\begin{prop}
Let $\bar M^2\times_\phi S^1$ be the twisted product of a closed surface $\bar M^2$ and a circle $S^1$.
Then $({\cal D},T)$, where $T$ is tangent to the fibers and ${\cal D}$ is tangent to the leaves, is critical for \eqref{E-gv-invar}
if and only if $\phi$ is the product of functions $\phi_1\in C^2(\bar M^2)$ and $\phi_2\in C^2(S^1)$.
\end{prop}

\begin{proof}
By conditions, the leaves $\bar M\times\{y\}$ are totally geodesic: $h=0$,
and the fibers $S^1\times\{y\}$ have constant curvature: $T(k)=0$.
Let $k\ne0$.
By Remark~\ref{R-wp}, a pair $({\cal D},T)$ is critical for \eqref{E-gv-invar} if and only if $\tau=0$.
By Frenet formula for $\nabla_T N$, the equality $\tau=0$ holds if and only if
the curvature vector $k N$ is parallel (in the normal connection) along $T$-curves.
By Lemma~\ref{L-chen1}, this is equivalent to $\phi=\phi_1\phi_2$.
\end{proof}

\begin{cor}
Let $\bar M^2\times_\phi S^1$ be a warped product. Then $({\cal D},T)$, where $T$ is tangent to the fibers and ${\cal D}$ is tangent to the leaves, is critical for \eqref{E-gv-invar} and $\gv({\cal D},T)=0$.
\end{cor}

\section{Jacobi operator}
\label{sec:jacobi}

The bilinear form $I_T$ depends on $T$ on $M$ with integrable ${\cal D}$; it serves as the index form of our variational problem . Let $g$ be any Riemannian metric compatible with $(T,\omega)$, and ${\rm d}\,V_g$ its volume form. It defines Hodge star operator on the space of differential forms, $\star_r:\Lambda^r(M)\to\Lambda^{3-r}(M)$, where $0\le r\le 3$.
We will not "decorate" $\star$  with $_r$ in what follows.

Then $D:=\star({\cal L}_T^2\,d): \Lambda^1(M)\to \Lambda^1(M)$ is a self-adjoint Jacobi type operator. It is interesting to study the kernel (i.e. Jacobi type fields) and spectrum of~$D$ for various critical pairs $(T,\omega)$ on $M$.

\begin{lem}\label{P-03}
Let $({\cal D},T)$ be a critical pair for \eqref{E-gv-invar} and ${\cal D}=\ker\omega$.
Then $I(\dot\eta,\dot\eta)\ge0$ for all variations of types $(i)$--$(ii)$ if and only if
\begin{equation}\label{E-condA}
 \omega\wedge d\omega\ge0\ ({\rm confoliation}),\quad \tau \wedge d\tau\ge0,\quad
 \star(\omega\wedge d\omega)\star(\tau \wedge d\tau) -(\star(\omega\wedge d\tau))^2 \ge0
\end{equation}
$($with respect to a given orientation$)$ for $\tau = \iota_{\dot X}\,d\omega$ and for all $\dot X\in\mathfrak{X}_{\cal D}$.
\end{lem}

\begin{proof} Let $T_t=e^{-f_t}T+X_t$ and $\omega_t=e^{f_t}\omega$
for a smooth functions $f_t$ on $M$ and $X_t\in\mathfrak{X}_{\cal D}$, where $f_0=0=X_0$ due to
Proposition~\ref{prop1}. Then $\dot T=-\dot f\,T+\dot X,\,\dot\omega=\dot f\,\omega$ and
\begin{eqnarray*}
 \dot\eta = T(\dot f)\,\omega-d\dot f +\iota_{\dot X}\,d\omega,\quad
 d\dot\eta = d(T(\dot f))\,\omega +T(\dot f)\,d\omega +d(\iota_{\dot X}\,d\omega).
\end{eqnarray*}
From this we find
\begin{eqnarray*}
 && \dot\eta\wedge d\dot\eta = T(\dot f)^2\omega\wedge d\omega + T(\dot f)\,\omega\wedge d(\iota_{\dot X}\,d\omega)
 -d\dot f\wedge d(T(\dot f))\wedge\omega -T(\dot f)\,d\dot f\wedge d\omega \\
 && - d\dot f\wedge d(\iota_{\dot X}\,d\omega) + (\iota_{\dot X}\,d\omega)\wedge d(T(\dot f))\wedge\omega
  +T(\dot f)(\iota_{\dot X}\,d\omega)\wedge d\omega +(\iota_{\dot X}\,d\omega)\wedge d(\iota_{\dot X}\,d\omega)
\end{eqnarray*}
or, in a bit simpler form, suitable for integration,
\begin{eqnarray*}
 \dot\eta\wedge d\dot\eta \eq T(\dot f)^2\omega\wedge d\omega
 + T(\dot f)\,\omega\wedge d(\iota_{\dot X}\,d\omega)\\
 && +\, (\iota_{\dot X}\,d\omega)\wedge d(T(\dot f))\wedge\omega
  + T(\dot f)(\iota_{\dot X}\,d\omega)\wedge d\omega
  +(\iota_{\dot X}\,d\omega)\wedge d(\iota_{\dot X}\,d\omega)\\
 && +\, d\alpha
\end{eqnarray*}
with $\alpha = \dot fd\dot\eta$,
and again, with $\beta = T(\dot f)\omega\wedge d(\iota_{\dot X}d\omega)$,
\begin{eqnarray*}
 \dot\eta\wedge d\dot\eta \eq T(\dot f)^2\omega\wedge d\omega
 + 2\,T(\dot f)\,\omega\wedge d(\iota_{\dot X}d\omega)+(\iota_{\dot X}\,d\omega)\wedge d(\iota_{\dot X}\,d\omega)\\
 && +\, d\alpha + d\beta.
\end{eqnarray*}
Replacing $f\to x_1f,\ X\to x_2X$ with arbitrary $x_1,x_2\in\RR$, we obtain
\[
 I(\dot\eta,\dot\eta)= \int_M\big\{
 x_1^2\,T(\dot f)^2(\omega\wedge d\omega) + 2\,x_1 x_2 T(\dot f)\,(\omega\wedge d\tau) + x_2^2(\tau\wedge d\tau)\big\}.
\]
Since $\dot f$ and $\dot X$ can be supported in a neighborhood of any point of $M$, $I(\dot\eta,\dot\eta)\ge0$ means
\begin{equation}\label{E-condB}
 x_1^2\,T(\dot f)^2(\omega\wedge d\omega) + 2\,x_1 x_2 T(\dot f)\,(\omega\wedge d\tau) + x_2^2(\tau\wedge d\tau)\ge0.
\end{equation}
Observe that \eqref{E-condB} is equivalent to conditions \eqref{E-condA}.
\end{proof}

\smallskip

The above leads to the following result for contact distributions.

\begin{thm}\label{prop2a} Let $\omega$ be a contact 1-form and $T$ its Reeb field on $M^3$. Then

\noindent\
$(a)$~$(\ker\,\omega, T)$ is a critical point for Godbillon-Vey integral \eqref{E-gv-invar};

\noindent\
$(b)$~there are no extremal with respect to \eqref{E-gv-invar} contact structures on $M^3$.
\end{thm}

\begin{proof}
By $\eta=0$, see Example~\ref{Ex-cont}, we get $d\eta=0$, and (a) follows from \eqref{E-gv-omegat} and Lemma~\ref{P-03}.
By \eqref{E-gv-omegatt}, $(\ker\,\omega, T)$ is a local minimum of \eqref{E-gv-invar} if and only~if \eqref{E-condA}~hold
(with opposite signs for maximum).

A contact structure on $M^3$ is given in some coordinates $(x_1,x_2,x_3)\in\RR^3$ by
 $\omega= dx_3-x_2 dx_1$.
Hence $d\omega=dx_1\wedge dx_2$ and $\omega(T)=1$, where $T=\partial_3$.
Set $X_1=\partial_1+x_2\partial_3$ and $X_2=\partial_2$.
Since $\omega(X_1)=\omega(X_2)=0$, we have $\ker\omega={\rm span}(X_1,X_2)$.
Next,
 $\iota_{X_1} d\omega = dx_2,\
 \iota_{X_2} d\omega = dx_1$
and $\eta=\iota_{T}\,d\omega = 0$.
Put $\dot X=p_1X_1+p_2X_2$ and find $dp_i=\sum_j\,p_{i,j}\,dx_j$. Thus
\[
 \tau=\iota_{\dot X}\,d\omega = p_1\,dx_2 +p_2\,dx_1,\quad
 d\tau = dp_1\wedge dx_2+dp_2\wedge dx_1.
\]
 Finally,
\begin{eqnarray*}
 \tau\wedge d\tau = -(p_2 p_{1,3}+p_1 p_{2,3})\,dx_1\wedge dx_2\wedge d x_3,\\
 \omega\wedge d\tau = (p_{1,1}-p_{2,2} +x_2 p_{1,3})\,dx_1\wedge dx_2\wedge d x_3.
\end{eqnarray*}
Note that $\omega\wedge d\omega = dx_1\wedge dx_2\wedge d x_3 >0$, see \eqref{E-condA}$_{1}$,
and  \eqref{E-condA}$_{2,3}$  can be written as
\begin{equation}\label{E-chart-cond}
 (p_{1,1}-p_{2,2} +x_2 p_{1,3})^2
 < -(p_2 p_{1,3}+p_1 p_{2,3}) = -(p_{1}p_{2})_{,3}
\end{equation}
and should be satisfied pointwise (within our chart) for any $p_1,p_2$.
Functions $p_1=1$ and $p_2=-x_3$ obey \eqref{E-chart-cond},
while functions $p_1=1$ and $p_2=x_3$ do not.
\end{proof}

\begin{remark}\rm
The condition $d\eta=0$ is not directly related to integrability of ${\cal D}$.
There exist foliations with non-zero $\gv$ class represented by a form $\eta\wedge d\eta$
which is positive (w.r.t. an orientation) at all points of the manifold.
And, there exist non-integrable distributions with $d\eta = 0$ (even, $\eta = 0$).
This condition ($d\eta = 0$) depends not only on ${\cal D}= \ker\omega$  but on both, $\omega$ and $T$ ($\omega(T)=1$), adapted to ${\cal D}$.
\end{remark}

\begin{example}[Operator $D=\star\,({\cal L}_T)^2d$ in coordinates]\rm
In~order to study our differential operator $D$, its spectrum and kernel (``Jacobi fields") in coordinates,
continue calculation of the above proof. Let $T=\partial_z$ on a domain in $M$ with coordinates $(x_1,x_2,z)$.
By $\omega(T)=1$ we have $\omega=P_1(x_1,x_2,z)\,dx^1+P_2(x_1,x_2,z)\,dx^2+dz$,
and ${\cal D}=\ker\omega$ is spanned by two vector fields $X_1=\partial_1-P_1\partial_z$ and $X_2=\partial_2-P_2\partial_z$.
 The integrability condition $\omega\wedge d\omega=0$ yields
\begin{equation}\label{E-ind-PP}
 P_2P_{1,3}-P_1P_{2,3}+P_{2,1}-P_{1,2}=0.
\end{equation}
 By condition of an extremum, see \eqref{E-ELiii}, we get
\begin{equation}\label{E-ind-b}
 ({\cal L}_T)^3\omega
 =P_{1,333}\,dx^1+P_{2,333}\,dx^2=0.
\end{equation}
By~\eqref{E-ind-b}, $P_{i,333}=0$,
thus $P_i$ in our coordinate system are the 2nd order polynomials in $z$:
 $P_i=\sum\nolimits_{j=0}^2 C_{ij}(x_1,x_2)z^j\ (i=1,2)$
with arbitrary functions $C_{ij}(x_1,x_2)$.
From \eqref{E-ind-PP}, we get two relations:
\[
 C_{20} = C_{10} C_{22}/C_{12},\quad C_{21} = C_{22} C_{11}/C_{12}.
\]
For a compatible metric $g$ we have $g(\partial_z,\partial_z)=1$, $g(X_i,\partial_z)=0$,
and $g(X_i,X_j)=d_{ij}(x_1,x_2,z)$ for $1\le i,j\le2$ with $d=d_{11}d_{22}-d_{12}^2>0$. Thus
\begin{equation*}
 g_{33}=1,\quad g_{i3}=P_i,\quad g_{ij}=d_{ij}+P_iP_j\quad (1\le i,j\le2),\quad
 \det g=d.
\end{equation*}
Let $\mu=\dot\omega=p_1(x_1,x_2,z)dx^1+p_2(x_1,x_2,z)dx^2$ be a variation of type (iii).
 Set $\alpha=({\cal L}_T)^2d\mu$.
Then $\star\,\alpha=(\star\,\alpha)_1\,dx^1+(\star\,\alpha)_2\,dx^2+(\star\,\alpha)_3\,dx^3$,
where $x_3=z$ and
\[
 g(\star\,\alpha,dx^j)=(\star\,\alpha)_i\,g(dx^i,dx^j)=(\star\,\alpha)_i\,g^{ij}
 \quad (1\le i,j\le 3).
\]
On the other hand, by $\alpha\wedge dx^i=g(\star\,\alpha,dx^i)\,{\rm d}V_g$ and \eqref{E-ind-a} we have
\[
 g(\star\,\alpha,dx^1)= -p_{2,333},\quad
 g(\star\,\alpha,dx^2)= -p_{1,333},\quad
 g(\star\,\alpha,dx^3)= p_{2,133}-p_{1,233}.
\]
Introducing functions $q_i=p_{i,33}$, we get the linear system with matrix $g^{-1}$, whose solution is
\begin{eqnarray*}
 && (\star\,\alpha)_1= -(d_{11}+P_1^2)\,q_{2,3} -(d_{12}+P_1P_2)\,q_{1,3}+P_1(q_{2,1}-q_{1,2}),\\
 && (\star\,\alpha)_2= -(d_{22}+P_2^2)\,q_{1,3} -(d_{12}+P_1P_2)\,q_{2,3}+P_2(q_{2,1}-q_{1,2}),\\
 && (\star\,\alpha)_3= q_{2,1}-q_{1,2} -P_1\,q_{2,3} -P_2\,q_{1,3}.
\end{eqnarray*}
Note that the following equalities hold:
\[
 (\star\,\alpha)_1 -P_1(\star\,\alpha)_3 = d_{11}\,q_{2,3}+d_{12}\,q_{1,3},\quad
 (\star\,\alpha)_2 -P_2(\star\,\alpha)_3 = d_{22}\,q_{1,3}+d_{22}\,q_{1,3}.
\]
Thus, and since $\mu$ has no $dz$ component, the eigenvalue problem
$D\mu:=\star\,\alpha=\lambda\mu$ for $D$ reads
\begin{equation}\label{E-ind2}
 d_{11}\,q_{2,3}+d_{12}\,q_{1,3} = -\lambda p_1,\quad
 d_{22}\,q_{1,3}+d_{22}\,q_{1,3} = -\lambda p_2,\quad
 q_{2,1}-q_{1,2} -P_1\,q_{2,3} -P_2\,q_{1,3}=0,
\end{equation}
where the third equation means compatibility, and the first two equations are equivalent to
\begin{equation}\label{E-ind3}
 p_{\,1,333333} = (\lambda^2/d) p_1,\quad
 p_{\,2,333333} = (\lambda^2/d) p_2.
\end{equation}
One may assume $d_{ij}(x_1,x_2,z)=\delta_{ij}$ without change of $\gv$ (since $T$ and ${\cal D}$ will not change), hence $d=1$.
The general solution of \eqref{E-ind3} in our coordinate system (when $d=1$) is
\begin{eqnarray*}
 && p_i = c_{i1}(x_1,x_2)\,{e^{\sqrt[3]{|\lambda|}\,z}} +c_{i2}(x_1,x_2)\,{e^{-\sqrt[3]{|\lambda|}\,z}} \\
 &&\quad +\,\big(c_{i3}(x_1,x_2)\,{e^{\frac12\sqrt[3]{|\lambda|}\,z}}
 +c_{i4}(x_1,x_2)\,{e^{-\frac12\sqrt[3]{|\lambda|}\,z}}\big)\cos((\sqrt{3}/2)\sqrt[3]{|\lambda|}\,z) \\
 &&\quad +\,\big(c_{i5}(x_1,x_2)\,{e^{\frac12\sqrt[3]{|\lambda|}\,z}}
 +c_{i6}(x_1,x_2)\,{e^{-\frac12\sqrt[3]{|\lambda|}\,z}}\big)\sin((\sqrt{3}/2)\sqrt[3]{|\lambda|}\,z)
\quad (i=1,2).
\end{eqnarray*}
The above functions $c_{ij}$ are related by \eqref{E-ind2}$_3$ and may be locally supported.
We omit further details about the spectrum of $D$, and will examine the case $\lambda=0$ only.

 The coefficient functions of "Jacobi fields" obey \eqref{E-ind3} with $\lambda=0$,
\[
 p_{\,i,333333} =0\ \Leftrightarrow\ q_{\,i,333} =0\quad (i=1,2).
\]
Thus
$p_i$ are the 5th degree polynomials in $z$:
 $p_i=\sum\nolimits_{j=0}^5 c_{ij}(x_1,x_2)z^j\ (i=1,2)$,
where again, $c_{ij}$ are related by \eqref{E-ind2}$_3$.
Vanishing of $z^i$-coefficients yields the system of five equations
\begin{eqnarray*}
 && C_{10}\,c_{23} +C_{20}\,c_{13} = 0,\\
 && 4\,C_{10}\,c_{24} +C_{11}\,c_{23} +4 C_{20}\,c_{14} +C_{21}\,c_{13} = 0,\\
 && 10\,C_{10}\,c_{25} +4 C_{11}\,c_{24} + C_{12}\,c_{23} +10 C_{20}\,c_{15} +4 C_{21}\,c_{14} + C_{22}\,c_{13} = 0,\\
 && 5\,C_{11}\,c_{25} +2 C_{12}\,c_{24} +5 C_{21}\,c_{15} +2 C_{22}\,c_{14} = 0,\\
 && C_{12}\,c_{25} +C_{22}\,c_{15} = 0.
\end{eqnarray*}
 We conclude (with the assistance of Maple program) that
 all ``Jacobi fields" $\dot\omega=\mu$  are represented by seven independent functions
 $c_{i,j}\ (i=1,2,\,j=0,1,2)$ and $c_{2,5}$ in two real variables,
 which might have local support. Other functions of two variables,
 $c_{i,j}\ (i=1,2,\,j=3,4)$ and $c_{1,5}$, are related by:
\begin{eqnarray*}
 && c_{13} = -10\,c_{25}\,C_{10}/C_{22},\quad
 c_{14} = -(5/2)c_{25}\,C_{11}/C_{22},\\
 && c_{23} = 10\,c_{25}\,C_{20}/C_{22},\quad
 c_{24} = (5/2)c_{25}\,C_{21}/C_{22},\quad
 c_{15} = -c_{25}\,C_{12}/C_{22},
\end{eqnarray*}
and depend on given functions $C_{ij}$ that may be supported anywhere.
The above property (i.e. the polynomial in $z$ structure of functions $p_i$ in the presentation of $\dot\omega=\mu$) holds in any coordinate system with $T=\partial_z$, while the functions may change.
\end{example}

\section{Concordance and homotopy}

It is well known (\cite{cc}, vol. I, Section 3.6, for example)  that the Godbillon-Vey
class of foliations  is invariant under the relation of concordance (in fact, cobordance).

 The relation of concordance of foliations is stronger than concordance of distributions
 in the space of distributions.
Recall that two codimension-one foliations $\calf_0$ and $\calf_1$
of a manifold $M$ are {\it concordant} when there exists a codimension-one foliation
$\calf$ of a 'cylinder' $M\times [0, 1]$ which is transverse to the boundary $M\times\{ 0, 1\}$ and
induces  $\calf_i$ on $M\times\{ i\}$, $i = 0, 1$. If $\calf$ is given by the equation
$\omega =0$ and $d\omega = \omega\wedge\eta$ on $M\times [0, 1]$, then
$\calf_i$ is given by $\omega_i = 0$ and $d\omega_i = \omega_i\wedge\eta_i$,
where $\omega_i = \phi_i^\ast\omega$, $\eta_i = \phi_i^\ast\eta$
and $\phi_i: M\to M\times [0, 1]$ is given by $\phi_i(x) = (x, i)$, $i = 0, 1$.
Since the maps $\phi_0$ and $\phi_1$ are
homotopic and $\eta_i\wedge d\eta_i = \phi_i^\ast (\eta\wedge d\eta)$, the
cohomology classes of 3-forms $\eta_i\wedge d\eta_i$, $i = 0, 1$, are equal.

\begin{definition}\rm
We shall say that two pairs $(\omega_i, T_i)$, $i = 0, 1$, consisting of 1-forms
$\omega_i$ and vector fields $T_i$ satisfying $\omega_i(T_i) = 1$ are {\it
concordant} when there exists a pair $(\omega, T)$ consisting of a 1-form $\omega$
and a vector field $T$ on $M\times [0, 1]$ such that
\[
 \omega (T) = 1,\quad
 \omega_i = \phi_i^\ast\omega,\quad
 \phi_{i\ast} (T_i(x)) = T(\phi_i(x))
\]
for all $x\in M$ and $i = 0, 1$, and $\phi_i: M\to M\times [0, 1]$ is given by $\phi_i(x) = (x, i)$, $i = 0, 1$.
\end{definition}

If $M$ ($\dim M = 3$)  is closed and oriented, then it is parallelizable,
so one can find  triples $(\omega_j)$ and $(T_j)$, $j = 1, 2, 3$, of 1-forms
and vector fields satisfying
$\omega_j(T_k) = \delta_{jk}$ for all $j$ and $k\in\{ 1, 2, 3\}$. These fields and
forms can be extended over $\tilde M = M\times [0, 1]$  and completed by another
vector field and another form, say $d/dt$ and $dt$, to get parallelizations
of $T\tilde M$ and $T^\ast\tilde M$. Take on  $M$ any pair $(\omega, T)$ satisfying
$\omega (T) = 1$ and write $\omega = \sum_if_i\omega_i$, $T = \sum_jh_j T_j$.
Assume that $\omega$ and $T$ are unit with respect to given parallelizations,
that is that $\sum_if_i^2 = \sum_jh_j^2 = 1$. The condition $\omega (T) = 1$
implies that $f_i = h_i$ for all $i$'s, that is such a pair is uniquely determined by a
map $f = (f_1, f_2, f_3):M\to S^2$. Since $S^2$ is contractible in $S^3$,
$f$ is homotopic to a constant map $f_0:M\to S^3$. A homotopy between $f$
and a constant map $f_0$, say $f_0 = (0,0,0,1)$ everywhere on $M$,
determines a pair $(\tilde\omega , \tilde T)$ on $\tilde M$ which coincides
with  $(\omega , T)$ on, say,  $M\times\{ 1\}$ and with $(d/dt, dt)$ on
$M\times \{ 0\}$. Since, obviously, any pair $(e^\phi\omega, e^{-\phi}T)$ is
concordant to $(\omega , T)$ and the relation of concordance is transitive,
we arrive at the following conclusion:
\medskip

{\it Any two pairs $(\omega , T)$, $(\omega' , T')$ satisfying $\omega (T) =
\omega'(T') = 1$ on a closed, oriented 3-manifold $M$ are concordant in our sense.}
\medskip

Certainly, one can find such pairs with different Godbillon-Vey invariants. For
example, on the unit tangent bundle $S\Sigma$ of a closed, oriented surface
$\Sigma$ of genus $> 1$, one has a foliation $\calf$ (arising to a pair like that)
with non-zero Godbillon-Vey class  (\cite{gv} or \cite[vol.~I, Example~1.3.14]{cc})
and a contact structure (defined, for example,
as the week-stable or week-unstable distribution of the geodesic flow on $\Sigma$
equipped with a Riemannian metric of constant, negative curvature)  arising
to a pair $(\omega, \xi)$ which consists of a contact form $\omega$ and its Reeb
field $\xi$ and has zero as its Godbillon-Vey invariant (see Theorem~\ref{prop2a}).
Finally, take into account the following, rather trivial, observation:
for any $f$, the systems $(\omega , T)$ and $(e^f\omega , e^{-f}T)$ are homotopic (therefore, cobordant and concordant as well)
but in general their $\gv$ classes are different.
 Therefore, unfortunately,

\smallskip

 {\it our Godbillon-Vey type invariant is not invariant under the concordance relation defined above}.

\begin{remark}\rm
Thurston's construction \cite{th12} of a family of smooth foliations $\{\calf_t\}_{t > 0}$ on the 3-sphere, for which $\gv(\calf_t) = t$, is obtained from the weak stable foliation starting with a punctured surface
and the leaves being weakly stable submanifolds of the geodesic flow.
Therefore, if $(T, N, B)$ is the Frenet frame of curves orthogonal to the leaves as in Section \ref{sec:RW-form}, then
$T$ corresponds to strongly unstable directions, while $N$ and $B$ can be determined from the Lie algebra description of $T^1(H^2)$.
 Due to \cite[Section~1.3.3]{pat}, one can define a contact 1-form $\alpha$ is whose characteristic (Reeb) flow $T'$ coincides with  the geodesic flow restricted on $T^1(S^3)$. Let ${\cal D}'$ be the distribution orthogonal (with respect to the Sasaki metric) to $T'$.Rotating $T$ in the plane ${\rm span}(T',T)$, we obtain a deformation (homotopy) from Thurston's construction $({\cal D}, T)$ to the contact structure
$({\cal D}', T')$. Consequently, $\gv({\cal D}, T)\ne 0$ changes
continuously  to $\gv({\cal D}', T') = 0$.
\end{remark}

\section{Around the Reinhart-Wood formula}
\label{sec:RW-form}

Let $g$ be a compatible metric and $\nabla$ its Levi-Civita connection.
Let the \textit{curvature} $k$ of $T$-curves be nonzero on an open set $U$ of~$M$.
Thus, the unit normal $N$, the binormal $B=T\times N$ and the \textit{torsion} $\tau$ of $T$-curves are defined on $U$.
We  get the Frenet formulae:
\begin{equation}\label{E-Frene}
 \nabla_T\, T=kN,\quad
 \nabla_T\, N=-kT +\tau B,\quad
 \nabla_T\, B=-\tau N.
\end{equation}
By the formula for the Levi-Civita connection, $k=g([N,T],T)$.
 Define the \textit{non-symmetric scalar second fundamental form} $h$ of ${\cal D}$ by
\begin{equation*}
 h_{X,Y} = g(\nabla_X Y, T),\quad X,Y\in {\mathfrak X}_{\cal D},
\end{equation*}
and denote by $\sigma_1 = h(N,N) + h(B,B)$ its trace, i.e. the mean curvature of ${\cal D}$.
The~non-self-adjoint \textit{shape operator} $A: {\cal D}\to{\cal D}$
is given by  $g(AX, Y) = h_{X,Y}$ for all $X,Y\in\mathfrak{X}_{\cal D}$.
The \textit{integrability tensor} of ${\cal D}$ (vanishing when ${\cal D}$ is tangent
to a foliation $\calf$) is given by
\[
 {\cal T}_{X,Y} = g([X,Y],T)/2 = (h_{X,Y}-h_{Y,X})/2.
\]

\begin{lem}\label{L-02}
The 1-form $\eta$, see \eqref{E-eta}, defining the class $\gv({\cal D}, T)$,
is given by
\begin{equation}\label{E-eta0}
 \eta=(\nabla_T\,T)^{\,\flat} = k\,N^\flat,
\end{equation}
while the 2-form $d\eta$ attains the following values on $U$:
\begin{equation}\label{E-d-eta}
 d\eta(N,B) = -2\Div( {\cal T}_{N,B}\cdot T),\quad
 d\eta(T,B) = k(\tau - h_{B,N}),\quad
 d\eta(T,N) = T(k) -k h_{N,N}.
\end{equation}
\end{lem}

\begin{proof}
Indeed, since $\eta (T) = 0$, $\eta=(\nabla_T\,T)^{\,\flat}$ is orthogonal to $\omega = g(T,\,\cdot\,)$
and for $X \in{\mathfrak X}_{\cal D}$ one has
\begin{equation*}
 (d\omega - \omega\wedge\eta)(T, X) = g(T, [T, X]) + g(\nabla_T\,T, X) =  g(\nabla_X\,T, T) = 0.
\end{equation*}
This shows that $d\omega - \omega\wedge\eta$ is orthogonal to the plane
$\omega\wedge\omega^\perp$. Thus, the required formula follows.

Alternatively, one may compute the values of $\eta$ using \eqref{E-eta}:
$\eta(T)=d\omega(T,T)=0$ and
\begin{eqnarray*}
 && \eta(N)=d\omega(T,N) = -g(\nabla_T\,N-\nabla_N\,T,\,T)
 =g(\nabla_T\,T,N) = k,\\
 && \eta(B)=d\omega(T,B) = -g(\nabla_T\,B-\nabla_B\,T,\,T)
 =g(\nabla_T\,T,B) = 0.
\end{eqnarray*}
As far as $d\eta$ is concerned, one has
\begin{eqnarray*}
 && d\eta (N, B) = N(\eta(B)) - B(\eta(N)) - \eta([N,B])\\
 &&\quad = - B(g(\nabla_TT, N)) - g(\nabla_TT, \nabla_NB - \nabla_BN) = - B(k) - k\,g(\nabla_NB,N).
\end{eqnarray*}
Differentiating $g([N, B], T)$ in the $T$-direction, after a lengthy calculation involving
the use of symmet\-ries of the curvature tensor $R$ for the second order derivatives $\nabla_T\nabla_N B$ and $\nabla_T\nabla_B N$, yields
\begin{equation*}
 T\,(g([N,B],T)) = B(k) + k\,g(\nabla_NB,N) +2\,\sigma_1{\cal T}_{N,B} .
\end{equation*} 
Notice that $\Div T=-\sigma_1$. From this and equality $2(\nabla_T\,{\cal T})_{N,B} = T(g([N,B],T))$ we deduce \eqref{E-d-eta}$_1$:
\begin{equation}\label{E-d-eta2}
 B(k) + k\,g(\nabla_NB,N) = 2(\nabla_T\,{\cal T})_{N,B} -2\,\sigma_1{\cal T}_{N,B}.
\end{equation}
Next,
\[
 d\eta(T,B)=T(\eta(B))-B(\eta(T))-\eta([T,B])=k\,g([T,B],N),
\]
from which \eqref{E-d-eta}$_{2}$ follows. The~proof of \eqref{E-d-eta}$_{3}$ is also straightforward.
\end{proof}

 Using \eqref{E-Frene}
 we find
\begin{equation}\label{eq:rwood1}
 \eta\wedge d\eta = - k^2(\tau - h_{B,N})\,{\rm d}V_g,
\end{equation}
${\rm d}V_g$ being the volume form on $(M, g)$. When $M$ is closed
(i.e. compact and without boundary),
identifying $H^3(M, \RR)$ with $\RR$ {\it via} form integration
we can arrive at the {\it Reinhart-Wood formula}
\begin{equation}\label{eq:rwood2}
 \gv({\cal D}, T) = -\int_M k^2(\tau - h_{B,N})\,{\rm d}V_g,
\end{equation}
which has been obtained for foliations in \cite{rw73} (with opposite sign because of our choice for $\gv({\cal D}, T)$).
If~${\cal D}$ is integrable, i.e. ${\cal D}=T\calf$, then
obviously $\gv({\cal D}, \tilde T) = \gv({\cal D}, T)=\gv(\calf)$ for any $\tilde T$ transverse to~${\cal D}$.

Writing (\ref{eq:rel_forms})--\eqref{eq:rel-forms3-gen} in terms of the Frenet frame $(T, N, B)$ one gets the following.

\begin{prop}\label{prop3}
Let $(T,\omega)$ and $(\tilde T,\tilde\omega)$ obey \eqref{E-wt=1} and $g\in{\rm Riem}(M,{\cal D},T)$ with ${\cal D}=\ker\omega$.

$(i)$~If $\,\tilde T = e^{-f} T$ for a smooth function $f$ on $M$, see Proposition~\ref{prop1}(i), then
\begin{align}\label{eq:rel_gv0}
 \gv({\cal D}, \tilde T) = \gv({\cal D}, T)
 - 2\int_M \big( 2\,T(f)\Div( {\cal T}_{N,B}\cdot T) + T(f)^2\,{\cal T}_{N,B} \big)\,{\rm d}V_g .
\end{align}

$(ii)$~If $\,\tilde T=T+X$ for some $X\in\mathfrak{X}_{\cal D}$, and $\tilde\omega=\omega$,
see Proposition~\ref{prop1}(ii), then
\begin{eqnarray}\label{eq:rel_gv2}
\nonumber
 \gv({\cal D}, \tilde T) \eq \gv({\cal D}, T) + 2\int_M\Big\{
  2\,k\,g(X,N)\Div( {\cal T}_{N,B}\cdot T) - 2(T(k) -k\,h_{N,N})\,{\cal T}_{X,B} \\
 \plus 2\,k(\tau - h_{B,N})\,{\cal T}_{X,N} - (Q_1+Q_2+Q_3)  \Big\}\,{\rm d}V_g ,
\end{eqnarray}
where the second order in $X$ terms $Q_i$ are given by
\begin{eqnarray*}
 Q_1 \eq (k^2/2)\,g(X,N)^2\,{\cal T}_{N,B} ,\\
 Q_2 \eq \big(T({\cal T}_{X,B}) -B((k/2)\,g(X,N)) -{\cal T}_{X,AB} +2\,\tau{\cal T}_{X,N}\big)\,{\cal T}_{X,N} , \\
 Q_3 \eq \big(N((k/2)\,g(X,N)) -T({\cal T}_{X,N}) -(k^2/2)\,g(X,N) +\tau{\cal T}_{X,B} +{\cal T}_{X,AN}\big)\,{\cal T}_{X,B}.
\end{eqnarray*}

$(iii)$~If $\,\tilde T=T$ and $\tilde\omega=\omega+\mu$ for some 1-form $\mu$,
see Proposition~\ref{prop1}(iii), then $\mu(T)=0$ and
\begin{eqnarray}\label{eq:rel_gv3}
 \nonumber
 \gv(\tilde{\cal D}, T) = \gv({\cal D}, T) \plus \int_M \big\{
 2\,(T(k) - k\,h_{N,N})(\<\nabla_T\,\mu^\sharp -A\mu^\sharp,\,B\> ) \\
 \minus 2\,k(\tau - h_{B,N})\,\<\nabla_T\,\mu^\sharp -A\mu^\sharp,\,N\> + Q\big\}\,{\rm d}V_g ,
\end{eqnarray}
where the second order in $\mu$ term $Q$ is given by
\begin{eqnarray*}
 Q \eq \<\nabla_T\,\mu^\sharp -A\mu^\sharp,\,N\>( d\mu(T,[T,B]) - T(d\mu(T,B)) ) \\
 \plus (\<\nabla_T\,\mu^\sharp -A\mu^\sharp,\,B\> )(T(d\mu(T,N)) - d\mu(T,[T,N])).
\end{eqnarray*}
For integrable ${\cal D}$, cases $(i)$--$(ii)$  reduce themselves to the expected equality $\gv({\cal D}, \tilde T) = \gv({\cal D}, T)$.
\end{prop}

\begin{proof}
(i) One has
 $d\omega (N, B) = - \omega ([N, B]) = - 2\,{\cal T}_{N,B}$.
The above, \eqref{E-d-eta}$_1$ and \eqref{eq:rel_forms} provide \eqref{eq:rel_gv0}.

(ii) Using the equalities
\begin{equation}\label{E-dd}
 d\omega(T,X) = k\,g(X,N),\quad d\omega(X,B) = -2\,{\cal T}_{X,B},\quad d\omega(X,N) = -2\,{\cal T}_{X,N},
\end{equation}
and \eqref{E-d-eta}$_{2,3}$, we derive the last two terms of \eqref{eq:rel-forms2-gen},
\begin{eqnarray*}
 d\eta\wedge \iota_X d\omega(T,N,B) \eq d\eta(T,N)\,d\omega(X,B) +d\eta(N,B)\,d\omega(X,T) +d\eta(B,T)\,d\omega(X,N) \\
 \eq 2\,k\,g(X,N)\Div( {\cal T}_{N,B}\cdot T) - 2(T(k) -k\,h_{N,N})\,{\cal T}_{X,B} +2\,k(\tau - h_{B,N})\,{\cal T}_{X,N} ,
\end{eqnarray*}
\begin{eqnarray*}
 (\iota_X d\omega\wedge d(\iota_X d\omega))(T,N,B) \eq d\omega(X,T)\,d(\iota_X d\omega)(N,B)
  + d\omega(X,N)\,d(\iota_X d\omega)(B,T)\\
 \plus d\omega(X,B)\,d(\iota_X d\omega)(T,N),
\end{eqnarray*}
where using
 $d(\iota_X d\omega)(N,B) = k\,g(X,N)\,{\cal T}_{N,B}$
and calculating $d(\iota_X d\omega)$ on pairs $(B,T)$ and $(T,N)$, we find
\begin{eqnarray*}
 && d\omega(X,T)\,d(\iota_X d\omega)(N,B) = 2 Q_1,\quad
 d\omega(X,N)\,d(\iota_X d\omega)(B,T) = 2 Q_2,\\
 && d\omega(X,B)\,d(\iota_X d\omega)(T,N) = 2 Q_3.
\end{eqnarray*}
Note that  $Q_i=0$ for all $i$'s when ${\cal D}$ is integrable.
The above provides
\begin{eqnarray}\label{eq:rel_forms2}
\nonumber
  && \tilde\eta\wedge d\tilde\eta - \eta\wedge d\eta = d(\eta\wedge\iota_X\,d\omega) +2(Q_1+Q_2+Q_3)\,{\rm d}V_g\\
  && 4(k\,g(X,N)\Div( {\cal T}_{N,B}\cdot T) -(T(k) -k\,h_{N,N})\,{\cal T}_{X,B} +k(\tau -h_{B,N})\,{\cal T}_{X,N})\,{\rm d}V_g .
\end{eqnarray}
From \eqref{eq:rel_forms2} and the Divergence Theorem the required \eqref{eq:rel_gv2} follows.

(iii)
Using
 $d\mu(T,X) = \<\nabla_T\,\mu^\sharp -A\mu^\sharp,\,X\>\ (X\in\mathfrak{X}_{\cal D})$,
as in step (ii), we calculate
\begin{eqnarray*}
 && (d\eta\wedge \iota_T\,d\mu)(T,N,B) = (T(k) - k h_{N,N})(\<\nabla_T\,\mu^\sharp -A\mu^\sharp,\,B\> )
 -k(\tau - h_{B,N})\,\<\nabla_T\,\mu^\sharp -A\mu^\sharp,\,N\> ,\\
 && (\iota_T\,d\mu\wedge d(\iota_T\,d\mu))(T,N,B) =
 d\mu(T,N)\,d(\iota_T\,d\mu)(B,T) + d\mu(T,B)\,d(\iota_T\,d\mu)(T,N) .
\end{eqnarray*}
Then, calculating $d(\iota_T(d\mu))$ on pairs $(T,N)$ and $(T,B)$,
\[
 d(\iota_T(d\mu))(T,N) = T(d\mu(T,N)) - d\mu(T,[T,N]),\ \
 d(\iota_T(d\mu))(B,T) =  d\mu(T,[T,B]) - T(d\mu(T,B)),
\]
we get $(\iota_T\,d\mu\wedge d(\iota_T\,d\mu))(T,N,B) = Q$.
From the above \eqref{eq:rel_gv3} follows.
\end{proof}

\begin{cor}\label{C-02}
Let $g\in{\rm Riem}(M,{\cal D},T)$ and ${\cal D}=\ker\omega$.

$(i)$~If $T_t=T+\phi_t T$ for some $\phi_t\in C^1(M)\ (|t|<\eps)$ and $\phi_0\equiv0$, then
\begin{eqnarray}\label{eq:rel_gv1}
\nonumber
 \gv({\cal D}, T_t)'_{\,|\,t=0} \eq 4\int_M T(\dot\phi)\Div( {\cal T}_{N,B}\cdot T)\,{\rm d}V_g \\
  \eq -4\int_M \dot\phi\,\Div(\Div( {\cal T}_{N,B}\cdot T)\cdot T)\,{\rm d}V_g .
\end{eqnarray}

$(ii)$~If $T_t=T+X_t$, $X_t\in\mathfrak{X}_{\cal D}\ (|t|<\eps)$ and $X_0=0$, then
\begin{equation}\label{eq:d_gv2}
 \gv({\cal D}, T_t)'_{\,|\,t=0} \!= 4\!\int_M \big\< k\Div( {\cal T}_{N,B}\cdot T) N -(T(k)-k\,h_{N,N})({\cal T}_{\centerdot\,,B})^\sharp
 \!+k(\tau - h_{B,N})({\cal T}_{\centerdot\,,N})^\sharp,\,\dot X \big\>\,{\rm d}V_g .
\end{equation}

$(iii)$~If $\omega_t=\omega + \mu_t\ (|t|<\eps)$, $\mu_t(T)=0$ and $\mu_0=0$, then
\begin{eqnarray}\label{eq:d_gv3}
 \gv({\cal D}_t,T)'_{\,|\,t=0} = 2\!\int_M\big\<( (\sigma_1{-}\tau)\psi_2 {-} T(\psi_2) -\psi_2 A^*) B
 -((\sigma_1{+}\tau)\psi_1 {-} T(\psi_1) -\psi_1\,A^*) N , \dot\mu^\sharp \big\>\,{\rm d}V_g ,
\end{eqnarray}
where $\psi_1=k(\tau - h_{B,N})$, $\psi_2=T(k) -k\,h_{N,N}$ and $A^*:{\cal D}\to{\cal D}$ is adjoint to $A$.

For integrable ${\cal D}$, cases $(i)$--$(ii)$ reduce themselves to the expected equality $\gv({\cal D}, T_t)'_{\,|\,t=0} = 0$.
\end{cor}

\begin{proof}
(i)~The first equality of \eqref{eq:rel_gv1} is provided by equalities
\begin{equation}\label{E-gv-first}
 \gv({\cal D},T_t)'_{\,|\,t=0} = -2\int_M T(\dot\phi)\,\omega\wedge d\eta
\end{equation}
and $(\omega\wedge d\eta)(T,N,B)=\omega(T)\,d\eta(N,B)=-2(\nabla_T\,{\cal T})_{N,B}$.
The second equality of \eqref{eq:rel_gv1} follows from the above, the Divergence Theorem and a general formula
\begin{equation}\label{E-divF}
  \Div(\dot\phi\,Q\,T)= T(\dot\phi) Q + \Div(Q\,T)\dot\phi.
\end{equation}
applied to $Q=\Div( {\cal T}_{N,B}\cdot T)$.

(ii) By the proof of Proposition~\ref{prop3}(ii), we have
\begin{equation*}
 \eta_t\wedge d\eta_t = \eta\wedge d\eta -d\alpha_t +2\,d\eta_t\wedge(\iota_{X_t}\,d\omega)
 +(\iota_{X_t} d\omega)\wedge d(\iota_{X_t}\,d\omega),
\end{equation*}
thus,
\begin{equation*}
 (\eta_t\wedge d\eta_t)'_{|\,t=0} = -d\,\dot\alpha +2\,d\eta\wedge\iota_{\dot X}(d\omega),
\end{equation*}
and \eqref{eq:d_gv2} follows directly from the above and \eqref{eq:rel_gv2}.

(iii)~Using
\begin{eqnarray*}
 && g(\nabla_T\,\mu^\sharp, N) = T(\mu(N)) - \tau\mu(B) ,\quad
 g(\nabla_T\,\mu^\sharp, B) = T(\mu(B)) + \tau\mu(N),\\
 && \int_M \psi_1\,T(\mu(N))\,{\rm d}V_g = \int_M (\sigma_1\psi_1 - T(\psi_1))\,\mu(N)\,{\rm d}V_g,\\
 && \int_M \psi_2\,T(\mu(B))\,{\rm d}V_g = \int_M (\sigma_1\psi_2 - T(\psi_2))\,\mu(B)\,{\rm d}V_g.
\end{eqnarray*}
and Proposition~\ref{prop3} we arrive at the claim of \eqref{eq:d_gv3}.
\end{proof}

\begin{cor}
If ${\cal D}$ is nowhere integrable then all critical points $({\cal D},T)$ of $\gv$ with respect to variations of $T$
preserving our almost product structure (case (i) of Corollary~\ref{C-02})
have the same type: are either maximum (when ${\cal T}_{N,B}<0$) or minimum (when ${\cal T}_{N,B}>0$).
\end{cor}

\begin{proof}
We deal with variations $T_t=T+\phi_t\,T=e^{-f_t}T$, where $\phi_0=0=f_0$, see Corollary~\ref{C-02}(i).
From $\phi_t= t\dot\phi+(t^2/2)\ddot\phi+o(t^2)$ we find $\dot\phi=-\dot f$ and $\ddot\phi=(\dot f)^2-\ddot f$.
Using \eqref{eq:rel_forms} we obtain
\[
 \gv({\cal D},T_t)''_{\,|\,t=0} = 2\int_M \big(T(\ddot f) \omega\wedge d\eta + T(\dot f)^2 \omega\wedge d\omega\big).
\]
By \eqref{E-gv-first} and \eqref{E-divF} we get
$\gv({\cal D}, T_t)''_{\,|\,t=0} = 2\int_M T(\dot\phi)^2 \omega\wedge d\omega$.
Let $({\cal D},T)$ be critical for $\gv$ and such variations. Then,
by~the proof of Corollary~\ref{C-02}(i), we~get
 $\gv({\cal D}, T_t)''_{\,|\,t=0} = - 4\int_M T(\dot\phi)^2\,{\cal T}_{N,B}\,{\rm d}V_g$.
By~conditions, $N$ and $B$ are globally defined and ${\cal T}_{N,B}\ne0$.
\end{proof}

\begin{prop}\label{cor2}
Let $T$ be a {geodesic vector field} on $(M^3,g)$, then $({\cal D},T)$ is critical and $\gv({\cal D},T)=0$.
In~particular, two-dimensional transversely oriented Riemannian foliations of 3-manifolds are criti\-cal points for Godbillon-Vey
integrals varying over all plane fields.
\end{prop}

\begin{proof}
Since $k=0$, by Lemma~\ref{L-02}, $\eta=0$, Thus, $d\eta=0$, and using of \eqref{E-gv-omegat} completes the proof of first claim.
If $\calf$ is a Riemannian foliation and $g$ is bundle-like, then
the normal vector field $T$ is geodesic (i.e., $\nabla_T\,T = 0$) and $\eta=0$, see Lemma~\ref{L-02} again.
Thus, $d\eta=0$, and $(T\calf, T)$ is critical by~\eqref{E-gv-omegat}.
\end{proof}

\begin{prop}
Condition \eqref{E-ELiii} $($for integrable ${\cal D})$ in geometrical terms reads
\begin{eqnarray}\label{E-LT3omega}
\nonumber
  T(T(k)) -k\nabla_T\,h_{N,N} -2\,T(k) h_{N,N}
  -k\,h_{AN,N} -k\,\tau^2=0,\\
  2\,T(k(h_{B,N}-\tau)) + k(-T(\tau) +h_{AB,N} +\tau(h_{AB,B}-h_{AN,N}) ) =0.
\end{eqnarray}
\end{prop}

\begin{proof}
Using $({\cal L}_T\,\xi)(y)=T(\xi(y)) - \xi([T,y])$ for any 1-form $\xi$, Frenet formulas \eqref{E-Frene},
definition of $h$ and \eqref{E-eta}, we get
\begin{eqnarray*}
  (({\cal L}_T)^3 \omega)(y) \eq (({\cal L}_T)^2\eta)(y) = T({\cal L}_T\,\eta(y)) - ({\cal L}_T\,\eta)([T,y]) \\
 \eq T(T(k\,g(N,y))) -k\,g(N,[T,y]) -T(k\,g(N,[T,y])) +k\,g(N,[T,[T,y]])\\
 \eq T(T(k\,g(N,y))) -2\,T(k\,g(N,[T,y])) + k\,g(N, [T, [T,y]]),
 \end{eqnarray*}
where one may assume $y\in{\cal D}$. For $y=N$ using $g(N, [T, N]) = h_{N,N}$, ${\cal T}_{N,B}=0$ and
\[
 g(N, [T, [T,N]]) = \nabla_T\,h_{N,N}
 -h_{AN,N} -\tau^2,
\]
this yields \eqref{E-LT3omega}$_1$, and for $y=B$
using $g(N, [T, B]) = h_{B,N} - \tau$ and
\[
 g(N, [T, [T,B]]) = -T(\tau) +h_{AB,N} +\tau(h_{AB,B}-h_{AN,N}),
\]
this yields \eqref{E-LT3omega}$_2$.
\end{proof}

\section{Variable Riemannian metric}
\label{sec:var-g-Riem}

Functional \eqref{eq:rwood2} leads to two functionals on the space of metrics ${\rm Riem}(M)$ on a manifold $M^3$
equipped with either a plane field ${\cal D}$ (then $T$ varies) or a unit vector field $T$ (then ${\cal D}$ varies).
Here we study the first of them.
So, let $(M,g)$ be a Riemannian manifold of dimension $3$ equipped with a plane field~${\cal D}$. We are looking for the first variation and critical points (Riemannian metrics) of the functional
\begin{equation}\label{E-rwood}
 J_{\cal D} : g \to -\int_M k^2(\tau - h_{B,N })\,{\rm d}V_g.
\end{equation}
The integrand is taken zero outside of $U = k^{-1} (\RR\smallsetminus \{ 0\})$, and the~integral is taken over $M$ if it converges; otherwise,
one integrates over an arbitrarily large, relatively compact domain $\Omega$ in $M$,
containing supports of variations $(g_t)$ with $g_{0}=g$.
Revcall again that for integrable ${\cal D}$ (i.e. tangent to a foliation $\calf$), the 3-form
\[
 \gamma = -k^2(\tau - h_{B,N})\,{\rm d}V_g
\]
represents the Godbillon-Vey class of $\calf$, see \cite{rw73}; hence, the functional $J_{\cal D}$ is constant in this case.

  Observe that equality
 $\tau - h_{B,N} =0$
 means that the  distribution ${\rm Span}(T,B)$ built of rectifying planes of $T$-curves is integrable.

Let $g_t\ (|t| < \eps)$ be a 1-parameter family of metrics.
Define a symmetric $(0,2)$-tensor $S$ by
\[
 S=\dot g.
\]
It has six independent components $S_{T,T},S_{T,N},S_{T,B},S_{N,N},S_{N,B},S_{B,B}$.
A family $g_t$ preserving metric on ${\cal D}$ is called $g^\bot$-\textit{variation}:
it has three components $S_{T,T},S_{T,N},S_{T,B}$.
 A family $g_t$ preserving orthogo\-nality of the distributions is called \textit{adapted variation}:
it has four components $S_{T,T},S_{N,N},S_{N,B},S_{B,B}$.
Therefore, an adapted $g^\bot$-{variation} has one component $S_{T,T}$ only.

\begin{thm}\label{T-main1}
Euler-Lagrange equations for \eqref{E-rwood} with respect to all variations of Riemannian metric are given on $U$ by
\begin{subequations}
\begin{eqnarray}\label{E-EL-1}
 && \Div(\Div( {\cal T}_{N,B}\cdot T)\cdot T) =0,\\
\label{E-EL-2}
 && \Div( {\cal T}_{N,B}\cdot T) - (T(\log k) - h_{N,N}){\cal T}_{N,B} = 0,\\
\label{E-EL-3}
 && (\tau - h_{B,N})\,{\cal T}_{N,B} = 0 ,
\end{eqnarray}
\end{subequations}
where $\sigma_1 = h(N,N) + h(B,B)$ is, as before, the trace of $h$.
\end{thm}

\begin{proof}
Arbitrary variation of a Riemannian metric can be decomposed into three cases:

1.~the metric varies along $T$ only;

2.~variations preserve the metric on ${\cal D}$ and $T$ (but disturb their orthogonality);

3.~the metric varies along ${\cal D}$ only (variations preserve the unit normal to~${\cal D}$).

\noindent
Thus, we divide the components of $S$ into three sets:
$\{S_{T,T}\}$, $\{S_{T,N},S_{T,B}\}$ and $\{S_{N,N},S_{N,B},S_{B,B}\}$.


Case 1.
Here, $T_t=e^{-f_t}\,T$ is the unit normal to ${\cal D }$ with respect to $g_t$ for some smooth function $f_t$ with $f_0=0$.
Differentiating $g_t(T_t,T_t)=1$ at $t=0$ we obtain $g(\dot f\,T,T)+S_{T,T}=0$. Hence,
 $\dot f = -S_{T,T}/2$.
 By Corollary~\ref{C-02}(i), we have
\begin{eqnarray*}
 \frac{{\rm d}}{{\rm dt}}\,J_{{\cal D}}(g_t)_{\,|\,t=0} \eq \int_M (\eta_t\wedge d\eta_t)'_{|\,t=0}
 = 2\int_M T(\dot f)\Div( {\cal T}_{N,B}\cdot T)\,{\rm d}V_g \\
 \eq -2\int_M \Div(\Div( {\cal T}_{N,B}\cdot T)\cdot T) S_{T,T}\,{\rm d}V_g.
\end{eqnarray*}
Thus, the Euler-Lagrange equations have the form of \eqref{E-EL-1}.
Note that $T((\nabla_T\,{\cal T})_{N,B})=(\nabla^2_{T,T}\,{\cal T})_{N,B}$.

Case 2.
Now, $T_t=T+X_t$ is the unit normal to ${\cal D}$ with respect to $g_t$
for some vector field $X_t\in\mathfrak{X}_{\cal D}$ with $X_0=0$.
Differentiating $g_t(T+X_t,N)=0$ at $t=0$ we obtain $g(\dot X,N)=-\dot g(T,N)=-S_{T,N}$. Similarly, we get $g(\dot X,B)=-S_{T,B}$.
Hence, $\dot X = -S_{T,N}N -S_{T,B}B$.
By Corollary~\ref{C-02}(ii) and using equalities ${\cal T}_{\dot X,N} = S_{T,B}{\cal T}_{N,B}$
 and ${\cal T}_{\dot X,B} = -S_{T,N}{\cal T}_{N,B}$, we find
\begin{eqnarray*}
 \frac{{\rm d}}{{\rm dt}}\,J_{{\cal D}}(g_t)_{\,|\,t=0}
 = 4\int_M \Big\{ k\,\dot g(X,N)\Div( {\cal T}_{N,B}\cdot T) -(T(k)-k h_{N,N}){\cal T}_{\dot X,B}
 +k(\tau - h_{B,N})\,{\cal T}_{\dot X,N}\Big\}\,{\rm d}V_g \\
 = 4\int_M \!\Big\{ \big((T(k)-k h_{N,N}){\cal T}_{N,B} -k\,\Div( {\cal T}_{N,B}\cdot T)\big)S_{T,N}
 -k(\tau - h_{B,N})\,{\cal T}_{N,B}\,S_{T,B}\Big\}\,{\rm d}V_g .
\end{eqnarray*}
Thus, the Euler-Lagrange equations (equivalent to vanishing of $S_{T,N},S_{T,B}$ components of the integrand) have the form of \eqref{E-EL-2} and \eqref{E-EL-3}.

Case 3. Since metric $g\in{\rm Riem}(M,{\cal D},T)$ can vary along ${\cal D}$ only, $\omega$ and $T$ do not change.
By~Lemma~\ref{E-eta}, $\eta$ does not vary; thus, by \eqref{eq:rwood1}, the functional $J_{\cal D}$ is constant.
These variations do not provide us with new Euler-Lagrange~equations.
\end{proof}

\begin{cor}\label{C-geod} Let $T$ be a geodesic vector field on $(M^3,g)$ and a normal plane field ${\cal D}$ be orthogonal to $T$.
 Then $g$ is a critical point for $J_{\cal D}$, but $g$ is not an extremum.
\end{cor}

\begin{proof}
Since $T$ is geodesic field, we have $k=0$. Hence \eqref{E-EL-2} and \eqref{E-EL-3} are satisfied.
By \eqref{E-d-eta2}, the Euler-Lagrange equation \eqref{E-EL-1} is satisfied, and the first claim follows.

 By \eqref{E-eta0}, $\eta=0$, hence $d\eta=0$, and we can apply Lemma~\ref{P-03}.
Assume $k\ne0$ at a point $x$ and then let $k\to0$.
Using \eqref{E-dd}, rewrite \eqref{E-condA}$_1$ as
\[
 \omega\wedge d\omega = -2\,{\cal T}_{N,B}\ge0.
\]
The equality \eqref{E-condA}$_2$ requires computation for $\tau=\iota_{\dot X}\,d\omega$ and $d\tau=d\iota_{\dot X}\,d\omega$:
\begin{eqnarray*}
 && (\tau\wedge d\tau)(T,N,B) =
       d\omega(\dot X,T)[B(d\omega(\dot X,N))+N(d\omega(\dot X,B))]\\
 && +\,d\omega(\dot X,N)[B(d\omega(\dot X,T))+T(d\omega(\dot X,B))]
    +  d\omega(\dot X,B)[N(d\omega(\dot X,T))+T(d\omega(\dot X,N))] \\
 && = 2\,k\,g(\dot X,N)[B({\cal T}_{\dot X,N})+N({\cal T}_{\dot X,B})] \\
 && +\,2\,{\cal T}_{\dot X,N}[B( k\,g(\dot X,N))+2 T({\cal T}_{\dot X,B})]
  +2\,{\cal T}_{\dot X,B}[N( k\,g(\dot X,B))+2 T({\cal T}_{N,B})].
\end{eqnarray*}
Letting $k\to0$ and $\dot X=\cos\phi N+\sin\phi B$ we get
\[
 (\tau\wedge d\tau)(T,N,B) \to 2\sin(2\,\phi)\,T(({\cal T}_{N,B})^2),
\]
which can be either positive or negative for different $\phi$ when $T({\cal T}_{N,B})\ne0$,
and just \eqref{E-condA}$_2$ is not satisfied.
\end{proof}

\begin{cor} Let $g$ be a critical metric for $J_{\cal D}$ whith ${\cal D}$ being  nowhere integrable.
Then the distribution ${\rm Span}(T,B)$ of rectifying planes to $T$-curves is integrable (on $U$) and $\eta\wedge d\eta =0$, hence
$J_{\cal D}(g)=0$.
\end{cor}

\begin{proof} Recall that $k\ne0$ on $U$. Since ${\cal T}_{N,B}\ne0$, by \eqref{E-EL-3} and Lemma~\ref{L-02} we obtain $d\eta(T,B)=0$.
Hence $N^\flat\wedge d\eta=0$. By this, $g([T,B],N)=0$ (hence ${\rm Span}(T,B)$ is integrable on $U$)
and $d\eta = N^\flat\wedge\alpha$ for some 1-form $\alpha$.
The last equality yields $\eta\wedge d\eta =k N^\flat\wedge N^\flat\wedge\alpha=0$.
\end{proof}

\begin{example}\rm
 By Theorem~\ref{T-main1}, if ${\cal D}$ is integrable, then all the metrics on $M$ are critical for the functional \eqref{E-rwood}, and $J_{\cal D}$ is constant.
There exist metrics on 3-manifolds endowed with non-integrable plane fields,
which are critical for \eqref{E-rwood}.
Indeed, let ${\cal D}$ be the plane field orthogonal to Hopf circles on $S^3$ with the standard metric $g$.
Since $k=0$ on $S^3$, the metric $g$ is critical for $J_{\cal D}$.
\end{example}

\begin{example}\rm
 Recall that a Riemannian metric $g$ on a contact manifold $(M^{3},\omega)$ with the Reeb field $T$
 (see Theorem~\ref{prop2a}) is \textit{associated} if there exists a $(1,1)$-tensor $\phi$ such that for all $X,Y \in {\mathfrak X}_M$
\begin{equation*}
 \omega(X) = g(T,X) , \quad d\omega(X,Y) = g(X,\phi(Y)),\quad \phi^2=-\id +\,\omega\otimes T .
\end{equation*}
The above $(\phi, T, \omega, g)$ is called a \textit{contact metric structure} on $M$, see \cite{b2010}.
The integral curves of $T$ are geodesics for the contact metric structure.
By Corollary~\ref{C-geod}, contact metrics on a contact manifold $(M^{3},\omega)$ are critical for the functional $J_{\cal D}$.
\end{example}

\section{Variable Randers metric}

Recall that a~{Finsler structure} on a manifold $M$ is a family of Minkowski norms $F_p$ in tangent spaces $T_pM$ depending smoothly on a point $p\in M$. The symmetric bilinear form
\[
 g_y(u,v)=\frac12\,\frac{\partial^2}{\partial s\partial t}F^2(y+su+tv)_{|s=t=0},\quad y\ne0,
\]
is positive definite.
An important non-Riemannian quantity (called the \textit{Cartan~torsion}) is the symmetric trilinear form
 $C_y(u,v,w)=\frac14\,\frac{\partial^3}{\partial r\,\partial s\,\partial t}\,F^2(y+ru+sv+tw)_{\,|\,r=s=t=0}\ (y\ne0)$.

Let a Finsler manifold $(M^3,F)$ be endowed with a transversely oriented plane field~${\cal D}$.
For any $p\in M$, there are two normal directions to ${\cal D}_p$, opposite when $F$ is {reversible}, see \cite{sh2}.
Let $T$ be a unit vector field orthogonal to ${\cal D}$.
Define a particular Riemannian metric $g$ on $M$, see \cite{rw3}, which is compatible with $({\cal D},T)$:
\[
 g:=g_T .
\]
 The~{Chern connection} $D^T$ is torsion free and 'almost metric'; it is determined~by
\begin{equation}\label{E-D-connect}
  g(D^T_u\,v, w) - g(\nabla_u\,v, w) =  C_T(D^T_w\,T,u,v) - C_T(D^T_u\,T,v,w) - C_T(D^T_v\,T,u,w),
\end{equation}
 where $u,v,w\in {\mathfrak X}_M$ and $\nabla$ is the Levi-Civita connection of $g$.

 A~codimension one foliation $\calf$ (with  the distribution ${\cal D}=T\calf$ being integrable)
is said to be \textit{Riemannian} if $T$-curves are $D^T$-geodesics, that is if  $D^T_T\,T=0$.

Proposition~\ref{cor2} can be extended for Finsler metrics as follows.

\begin{cor}
Let $\calf$ be a 2-dimensional transversely oriented Riemannian foliation of $M^3$ with a Finsler metric $F$.
Then a pair $(T\calf,T)$, where $T$ a unit $F$-normal to $\calf$,
is a critical point for Godbillon-Vey integrals varying over all plane fields.
\end{cor}

\begin{proof}
Observe that $k^T=F(D^T_T\,T)$ is the curvature of $T$-curves on $(M,F)$.
By~(\ref{E-D-connect}), $D^T_T\,T = \nabla_T\,T$; thus,
$T$ is a geodesic vector field for $F$ if and only if it is geodesic for~$g$.
Hence, the claim follows from Proposition~\ref{cor2} for metric $g$.
\end{proof}

Important 'computable' examples of Minkowski norms are  {\it Randers norms}, that is  Euclidean norms shifted by a small vector.
These norms were introduced in \cite{ra} by a physicist G.\,Randers to consider the unified field theory.
 Given a Riemannian metric $a(\cdot\,,\cdot)=\<\cdot\,,\cdot\>$ on $M^3$ with Euclidean norm $\alpha(y):=\sqrt{\<y\,,y\>}$
and a~linear form $\beta$ on $M$ of norm $b:=\alpha(\beta)<1$, the Randers metric is defined by $F=\alpha+\beta$.
Set $c=\sqrt{1-b^2}$.
Let ${\cal D}$ be a plane field on $M$ and $\bar T$ an $a$-unit $a$-normal to~${\cal D}$.
For simplicity assume that $\beta^\sharp\in{\cal D}$, i.e.
\begin{equation*}
 \beta(\bar T)=0.
\end{equation*}

\begin{lem}[see \cite{rw3}]\label{L-Frenet}
Frenet frames $\{T,N,B\}$ and $\{\bar T,\bar N,\bar B\}$ of normal curves in metrics $g$ and $a$ $($defined on an open set
where the curvatures $k$ and $\bar k$ are nonzero$)$ are related as follows:
\begin{eqnarray*}
 && T = c^{-1}\,\bar T-c^{-2}\beta^\sharp,\quad
    c^2k N = \bar k\bar N -c^{-1}\bar\nabla^\top c +c^{-2}\beta(\bar k\bar N-c^{-1}\bar\nabla^\top c)\beta^\sharp.
\end{eqnarray*}
\end{lem}

In the case of $F=\alpha+\beta$, the~new metric $g=g_T$ has the form
\begin{equation}\label{E-c-value0}
 g = (1+\beta(n))\,a +\beta^\flat\otimes\beta^\flat -\beta(n)\,n^\flat\otimes n^\flat +\beta^\flat\otimes n^\flat +n^\flat\otimes\beta^\flat,
\end{equation}
where the vector field $n=c^2 T$ has the properties $\<n,n\>=1$ and $\beta(n)=-b^2$.
In particular,
\begin{equation}\label{E-c-value1}
 g(u,v) = c^2(\<u,v\> -\beta(u)\,\beta(v))\quad (u,v\in{\cal D}).
\end{equation}
 By Lemma~\ref{L-Frenet}, Riemannian foliations of $M^3$ with Randers metric $F$ are characterized by condition
\[
 \bar k\bar N -c^{-1}\bar\nabla^\top c +c^{-2}\beta(\bar k\bar N-c^{-1}\bar\nabla^\top c)\beta^\sharp=0.
\]
 Put $N=N_1\bar N+N_2\bar B$ and $B=B_1\bar N+B_2\bar B$.
 Then the components $N_1,N_2$ can be extracted using Lemma~\ref{L-Frenet}.
For example, $c^2k N=\bar k\bar N+c^{-2}\bar k\beta(\bar N)\beta^\sharp$ when $c=\const$.
We can use \eqref{E-c-value1} to write down the system for $B$:
\begin{equation*}
  0=c^{-2}g(B,N) = \<B,N\> -\beta(B)\beta(N),\qquad
  c^{-2}=c^{-2}g(B,B) = \<B,B\> -\beta(B)^2 .
\end{equation*}
This yields the linear system for components $B_1$ and $B_2$,
from which, keeping in mind $\beta(N)=N_1\beta(\bar N)+N_2\beta(\bar B)$, we find
 $B_1 = \frac{\sqrt{1+c^2}\,(\beta(\bar B)\beta(N)-N_2)}{c\,(N_1\beta(\bar B)-N_2\beta(\bar N))}$
 and
 $B_2 = \frac{\sqrt{1+c^2}\,(\beta(\bar N)\beta(N)-N_1)}{c\,(N_2\beta(\bar N)-N_1\beta(\bar B))}$.
The curvature and torsion functions $k,\tau$ of $T$-curves and the second fundamental form and integrability tensor $h,{\cal T}$ of ${\cal D}$ for $g$ are related with such quantities $\bar k,\bar\tau,\bar h,\bar{\cal T}$ for $a$ by long formulas~\cite{rw3}.

 Here, we are looking for Euler-Lagrange equations for the following functional on the space of Randers metrics on $(M,{\cal D})$:
\begin{equation}\label{E-rwood-Rand}
 J^R_{\cal D}: (\alpha,\beta) \to -\int_M k^2(\tau - h_{B,N})\,{\rm d}V_g\,.
\end{equation}
(Certainly,  for $\beta=0$ \eqref{E-rwood-Rand} reduces itself to  \eqref{E-rwood}.)
So, let $F_t=\alpha_t+\beta_t\ (|t| < \eps)$ be a family of Randers metrics, and $\bar T_t$ unit $\alpha_t$-normals to ${\cal D}$.

\begin{prop}
$(i)$~Let $\beta_t=\beta$, i.e. $F_t=\alpha_t+\beta$. Then for $u,v\in{\cal D}$, we have
\begin{subequations}
\begin{eqnarray}
\label{E-dotg-uv}
\nonumber
 \dot g(u,v) \eq c^2\dot a(u,v) -(c\,\dot a(\bar T,\beta^\sharp)+\dot a(\beta^\sharp,\beta^\sharp))\big(\<u,v\> -\beta(u)\beta(v)\big)\\
 && -\,2\,c^2\big(\dot a(\beta^\sharp,u)\beta(v) +\dot a(\beta^\sharp,v)\beta(u)\big),\\
\label{E-dotg-nv}
 \dot g(T,v) \eq c\,\dot a(\bar T,v) -\dot a(\beta^\sharp,v)
 +\beta(v)\big(\frac32\,\dot a(\beta^\sharp,\beta^\sharp)-2c\,\dot a(\bar T,\beta^\sharp)
 -\frac12\,{c^2}\,\dot a(\bar T,\bar T)\big),\\
\label{E-dotg-nn}
 \dot g(T,T) \eq \dot a(\bar T,\bar T) -2c^{-1}\dot a(\bar T,\beta^\sharp) +c^{-2}\dot a(\beta^\sharp,\beta^\sharp).
\end{eqnarray}
\end{subequations}
 $(ii)$~Let $\alpha_t=\alpha$, i.e. $F_t=\alpha+\beta_t$, and $\beta_t^\sharp\in{\cal D}$. Then for $u,v\in{\cal D}$, we have
\begin{subequations}
\begin{eqnarray}\label{E-dotg-uv-beta}
 \dot g(u,v) \eq -2\<\dot\beta,\beta\>\big(\<u,v\> -\beta(u)\beta(v)\big)
  -c^2\big(\dot\beta(u)\beta(v)+\dot\beta(v)\beta(u)\big),\\
\label{E-dotg-nv-beta}
 \dot g(T,v) \eq \dot\beta(v) -c\beta(v)\dot\beta(\bar T),\\
\label{E-dotg-nn-beta}
 \dot g(T,T) \eq 2\,c^{-1}\dot\beta(\bar T) -2\,c^{-2}\<\dot\beta,\beta\>.
\end{eqnarray}
\end{subequations}
\end{prop}

\proof
(i)~Derivating  \eqref{E-c-value0} we obtain
\begin{eqnarray}\label{E-dotg-a}
\nonumber
 \dot g(u,v) \eq (1+\beta(n))\dot a(u,v) +\beta(\dot n)\big(\<u,v\>-\<n,u\>\<n,v\>\big)\\
\nonumber
 && -\,\beta(n)\big(\dot a(n,u)\<n,v\> +\dot a(n,v)\<n,u\> +\<\dot n,u\>\<n,v\>+\<\dot n,v\>\<n,u\>\big)\\
 && +\,\beta(u)\big(\dot a(n,v)+\<\dot n,v\>\big) +\beta(v)\big(\dot a(n,u)+\<\dot n,u\>\big).
\end{eqnarray}
After derivation of $a_t(\bar T_t,u)=0\ (u\in{\cal D})$ and $a_t(\bar T_t,\bar T_t)=1$
we find
 $\dot{\bar T}^\top=-\dot a(\bar T,\,\cdot)^\sharp$
 and
 $\dot{\bar T}^\bot=-(1/2)\,\dot a(\bar T,\bar T)\bar T$,
respectively. Hence,
\[
 \dot{\bar T} = -\dot a(\bar T,\,\cdot)^\sharp -(1/2)\,\dot a(\bar T,\bar T)\bar T.
\]
Since
 $\dot n=\dot c\,\bar T + c\,\dot{\bar T} -\dot a(\beta^\sharp,\,\cdot)^\sharp$
 and
 $\dot c=-\dot a(\beta^\sharp,\beta^\sharp)/(2\,c)$,
then
\begin{eqnarray*}
 && \beta(\dot n) = c\,\beta(\dot{\bar T}) -\dot a(\beta^\sharp,\beta^\sharp) = -c\,\dot a(\bar T,\beta^\sharp)-\dot a(\beta^\sharp,\beta^\sharp),\\
 &&
 \<\dot n,u\> = \dot c\<\bar T,u\> + c\<\dot{\bar T},u\> -\dot a(\beta^\sharp,u).
\end{eqnarray*}
In particular, from \eqref{E-dotg-a} with $u,v\in{\cal D}$, using $\<\bar T,u\>=0$, we obtain
\begin{eqnarray*}
  \dot g(u,v) \eq (1+\beta(n))\dot a(u,v) -(c\,\dot a(\bar T,\beta^\sharp)+\dot a(\beta^\sharp,\beta^\sharp))\big(\<u,v\>-\<n,u\>\<n,v\>\big)\\
 && +\,\beta(n)\big((c\,\dot a(\bar T,u)-\dot a(\beta^\sharp,u))\beta(v) +(c\,\dot a(\bar T,v)-\dot a(\beta^\sharp,v))\beta(u) \\
 && -(c\,\dot a(\bar T,u)+\dot a(\beta^\sharp,u))\beta(v) -(c\,\dot a(\bar T,v)+\dot a(\beta^\sharp,v))\beta(u)\big)\\
 && +\,\beta(u)\big( (c\,\dot a(\bar T,v)-\dot a(\beta^\sharp,v)) -(c\,\dot a(\bar T,v)+\dot a(\beta^\sharp,v))\big) \\
 && +\beta(v)((c\,\dot a(\bar T,u)-\dot a(\beta^\sharp,u)) -(c\,\dot a(\bar T,u)+\dot a(\beta^\sharp,u))\big),
\end{eqnarray*}
and then \eqref{E-dotg-uv}.
Similarly, from \eqref{E-dotg-a} with $u=n$ and $v\in{\cal D}$, using $\<\dot n,n\>=\frac12\,\dot a(\beta^\sharp,\beta^\sharp)-\frac32\,c^2\dot a(\bar T,\bar T)$, we obtain
\begin{eqnarray*}
 \dot g(n,v) \eq c^2[\dot a(n,v) +\beta(v)(\dot a(n,n) +\<\dot n,n\>)]
\end{eqnarray*}
and then \eqref{E-dotg-nv}.
Finally, from \eqref{E-dotg-a} with $u=n$ and $v=n$ we get $\dot g(n,n)=c^{2}\dot a(n,n)$, hence~\eqref{E-dotg-nn}.

(ii)~In this case, $(\dot\beta)^\sharp=(\beta^\sharp)'$ and $\bar T_t=\bar T$.
Since
 $\dot n=\dot c\,\bar T -\dot\beta^{\,\sharp}$
 and
 $\dot c= -c^{-1}\<\dot\beta,\beta\>$,
then
\[
 \beta(\dot n) = -\<\dot\beta,\beta\>,\quad
 \dot\beta(n) = c\,\dot\beta(\bar T) -\<\dot\beta,\beta\>,\quad
 \<n,u\> = -\beta(u),\quad \<\dot n,u\> = -\dot\beta(u).
\]
Derivating  \eqref{E-c-value0} in this case yields
\begin{eqnarray}\label{E-dotg-b}
\nonumber
 \dot g(u,v) \eq -2\<\dot\beta,\beta\>[\<u,v\> -\<n,u\>\<n,v\>]
 +\dot\beta(u)\beta(v)+\beta(u)\dot\beta(v)\\
\nonumber
 \plus (1-c^2)\big(\<\dot n,u\>\<n,v\> +\<\dot n,v\>\<n,u\>\big) \\
 \plus \dot\beta(u)\<n,v\> + \beta(u)\<\dot n,v\> +\beta(v)\<\dot n,u\> +\dot\beta(v)\<n,u\> .
\end{eqnarray}
For $u,v\in{\cal D}$ this reduces to \eqref{E-dotg-uv-beta}.
From \eqref{E-dotg-b} for $u=n$ and $v\in{\cal D}$, using
 $\<\dot n,n\> = -c\,\dot\beta(\bar T)$,
we find \eqref{E-dotg-nv-beta}.
From \eqref{E-dotg-b} for $u=n$ and $v=n$ we find $g(n,n) = 2\,c^{3}\dot\beta(\bar T) -2\,c^2\<\dot\beta,\beta\>$, hence \eqref{E-dotg-nn-beta}.
\qed

\smallskip

Let $a_t\ (|t| < \eps)$ be a 1-parameter family of metrics and $\beta_t\ (|t| < \eps)$ a 1-parameter family of 1-forms.
Then $\dot a$ has six independent components:
$\dot a(\bar T,\bar T),\dot a(\bar T,\bar N),\dot a(\bar T,\bar B),\dot a(\bar N,\bar N),\dot a(\bar N,\bar B)$ and $\dot a(\bar B,\bar B)$, and $\dot\beta^{\,\sharp}$ has three independent components: $\dot\beta(\bar T),\dot\beta(\bar N)$ and $\dot\beta(\bar B)$.
Let $D^T_T\,T\ne0$ on an open set $U$.
Denote by $Q_i\ (i=1,2,3)$ the LHS's of equations \eqref{E-EL-1}--\eqref{E-EL-3} calculated with respect to new metric $g=g_T$.
The following result can be viewed as extension of Theorem~\ref{T-main1}.

\begin{cor}\label{T-main2}
Euler-Lagrange equations for \eqref{E-rwood-Rand} with respect to all variations $(\alpha_t,\beta_t)$ of metric $a$
and 1-form $\beta$ such that $\beta^\sharp\in {\cal D}$ are trivial on $M\setminus U$ and are given on $U$ by
\begin{equation}\label{E-EL-Rand}
 Q_1=0,\quad Q_2=0,\quad Q_3=0,
\end{equation}
which reduce to \eqref{E-EL-1}--\eqref{E-EL-3} when $\beta=0$.
\end{cor}

\proof Let $\beta\ne0$ (otherwise we reduce all of that to Theorem~\ref{T-main1} for metric $a$).
Due to Theorem~\ref{T-main1}, only the derivatives $\dot g(T,N),\dot g(T,B)$ and $\dot g(T,T)$ are important for Euler-Lagrange equations; hence, we may and do assume $\dot g(N,N)=\dot g(N,B)=\dot g(B,B)=0$. By the results of
Section~\ref{sec:var-g-Riem},
\begin{eqnarray}\label{E-JRD}
 \frac{{\rm d}}{{\rm dt}}\,J^R_{{\cal D}}(a_t,\beta_t)_{\,|\,t=0} \eq \int_M \big\{
 Q_1\,\dot g(T,T) +Q_2\,\dot g(T,N) +Q_3\,\dot g(T,B) \big\}\,{\rm d}V_g .
\end{eqnarray}
 If $\beta$ does not depend on $t$ then \eqref{E-JRD} and \eqref{E-dotg-nv}--\eqref{E-dotg-nn} yield
\begin{eqnarray*}
 2 c^{-2}Q_1 -\beta(N)Q_2 -\beta(B)Q_3 \eq 0,\\
 2c^{-2}\beta(\bar N)Q_1 -(N_1 -2\beta(N)\beta(\bar N))Q_2 -(B_1 -2\beta(B)\beta(\bar N))Q_3 \eq 0,\\
 2c^{-2}\beta(\bar B)Q_1 -(N_2 -2\beta(N)\beta(\bar B))Q_2 -(B_2 -2\beta(B)\beta(\bar B))Q_3 \eq 0.
\end{eqnarray*}
 Next, if $\alpha$ does not depend on $t$ then
\eqref{E-JRD} and \eqref{E-dotg-nv-beta}--\eqref{E-dotg-nn-beta} yield
\begin{eqnarray}
\label{E-EL-R5-temp}
 2c^{-2}\beta(\bar N) Q_1 - N_1 Q_2 - B_1 Q_3 = 0,\quad
 2c^{-2}\beta(\bar B) Q_1 - N_2 Q_2 - B_2 Q_3 = 0 .
\end{eqnarray}
Since  $\beta(\bar N)^2+\beta(\bar B)^2\ne0$,
we get \eqref{E-EL-Rand} directly from \eqref{E-EL-R5-temp}.
\qed

\smallskip

By Corollaries~\ref{C-geod} and \ref{T-main2}, a metric $g$ on $(M^3,{\cal D})$ with a geodesic vector field $T$
is critical for~$J^R_{\cal D}$.

\section{Final remarks}

In the case of a codimension-one distribution ${\cal D}$ and a vector field $T$ transverse to ${\cal D}$
on a manifold $M$ of $\dim M > 3$, the form $\eta\wedge d\eta$ can be also defined as in Section~\ref{sec:1-1}
for the $\dim M = 3$ case but there is no reason
(different from integrability of ${\cal D}$) for $\eta\wedge d\eta$ to be closed.
However, using the Hodge decomposition theorem:
 $\Lambda^k (M) = {\cal H}^k(M) \oplus \im d\oplus\im\delta$,
(see, for example, \cite{war}),
${\cal H}^k (M)$ being the space of harmonic $k$-forms on $M$ and
$\delta$  being the formal adjoint of $d$, one can project our form $\eta\wedge
d\eta$ onto the space ${\cal H}^3(M)$ and define the Godbillon-Vey class of
$({\cal D},T)$ as the cohomology class determined by this projection.
In this case, however, differently from the case of $\dim M=3$ (Lemma~\ref{L-eta}),
the
obtained class $\gv ({\cal D},T)$ depends strongly on the choice of a metric $g\in{\rm Riem}(M,{\cal D},T)$.
Anyway, it seems to be interesting to investigate this general case more~closely.

Let $\dim M=2n+1\ge5$, and $\omega,T$ and $\eta = \iota_{\,T}\,d\omega$ be as above.
Following the ideas of \cite{fh}, we observe that the following  cohomolo\-gy classes
({\it Godbillon-Vey type invariants}) are well-defined:
\begin{equation*}
 \gv_s(\omega,T) = \int_M \eta\wedge\underset{s}{\underbrace{ d\eta \wedge \ldots \wedge d\eta}}
 \,\wedge\,\underset{n-s}{\underbrace{d\omega \wedge \ldots \wedge d\omega}},\qquad
 0\le s\le n.
\end{equation*}
 Let $\{N,Z_0=B,Z_1,\ldots,Z_{2n-2}\}$ be a local orthonormal basis of ${\cal D}$,
 and as before, $h$ its second fundamental form, and $k,\tau$ the curvature and torsion of $T$-curves.
Let ${\cal T}^Z$ be the integrability tensor of the distribution ${\cal D}_Z$ orthogonal to $\{T,N\}$.
Denote by $S^i_{2n-2}\ (0\le i\le 2n-1)$ the set of all transpositions ${\bm j}=\{j_1,j_2\ldots,j_{\,2n-2}\}$ of $2n-2$ elements $\{0,1,\ldots,2n-2\}\setminus\{i\}$.

\begin{prop}
 We have $\gv_0(\omega,T) = 0$ and for $s\ge1$
\begin{eqnarray*}
 && \gv_s(\omega,T) = (-2)^{n-1}\int_M \big\{k^{s+1}\sum\nolimits_{i=0}^{2n-2}\big[(h_{Z_i,N}-\tau\delta_{i0})\times \\
 && \times \sum\nolimits_{{\bm j}\in S^i_{2n-2}}
 \underset{s}{\underbrace{ g({\cal T}^Z_{j_1,j_2},N)\ldots g({\cal T}^Z_{j_{\,2s-1},j_{\,2s}},N)}}\cdot
 \underset{n-s}{\underbrace{{\cal T}_{j_{\,2s+1},j_{\,2s+2}}\ldots {\cal T}_{j_{\,2n-3},j_{\,2n-2}} }}\big]\big\}\,{\rm d}V_g.
\end{eqnarray*}
 If ${\cal D}_Z$ is integrable, then $\gv_s(\omega,T) = 0$ for all $s\ge1$.
\end{prop}

\begin{proof} Since $\eta(T)=0,\,\eta(N)=k$ and $d\omega(T,Z_i)=0$, see \eqref{E-dd}, we get $\gv_0(\omega,T) = 0$.
We have
 $d\omega(Z_i,Z_j)=  g(T, [Z_i, Z_j])$
 and
 $d\eta(Z_i,Z_j)= g(N, [Z_i, Z_j])$.
From this the claim for $s\ge1$ follows.
\end{proof}

\end{document}